\documentclass{article}

\usepackage{amsmath,amssymb,amsthm}

\usepackage[T1]{fontenc}
\usepackage{kpfonts, baskervald}
\newtheorem{thm}{Theorem}
\numberwithin{thm}{section}
\newtheorem{cor}[thm]{Corollary}
\newtheorem{prop}[thm]{Proposition}
\newtheorem{lem}[thm]{Lemma}

\theoremstyle{definition}
\newtheorem{defn}[thm]{Definition}
\theoremstyle{remark}
\newtheorem{rmk}[thm]{Remark}
\newtheorem{exam}[thm]{Example}

\usepackage{microtype}

\usepackage{tikz-cd}

\usepackage{todonotes}

\usepackage{hyperref}
\hypersetup{
  colorlinks   = true, 
  urlcolor     = blue, 
  linkcolor    = blue, 
  citecolor   = red 
}

\newcommand{\co}{\colon\thinspace}
\newcommand{\mb}[1]{\mathbb{#1}}

\newcommand{\mc}[1]{\mathcal{#1}}
\newcommand{\mi}[1]{\mathscr{#1}}

\newcommand{\op}{{op}}

\newcommand{\FMap}{\mathrm{Map}^f}
\newcommand{\der}{\dag}

\DeclareMathOperator{\Hom}{Hom}
\DeclareMathOperator{\Map}{Map}
\DeclareMathOperator{\map}{map}

\DeclareMathOperator{\Fun}{Fun}

\DeclareMathOperator{\im}{Im}
\DeclareMathOperator{\ev}{ev}
\DeclareMathOperator{\Sp}{Sp}
\DeclareMathOperator{\sk}{sk}
\DeclareMathOperator{\sh}{sh}
\DeclareMathOperator{\shad}{sh^\perp}
\DeclareMathOperator{\gr}{gr}
\DeclareMathOperator{\Spaces}{\mi{S}}
\DeclareMathOperator{\Set}{Set}
\DeclareMathOperator{\sSet}{Set^\Delta}
\DeclareMathOperator{\cSet}{Set^\square}

\DeclareMathOperator{\Alg}{Alg}
\DeclareMathOperator{\Cat}{Cat}

\DeclareMathOperator{\LMod}{LMod}

\DeclareMathOperator{\Comod}{Comod}

\DeclareMathOperator{\Syn}{Syn}
\DeclareMathOperator{\SSyn}{SSyn}
\DeclareMathOperator{\Rees}{Rees}

\DeclareMathOperator*{\colim}{colim}
\DeclareMathOperator*{\hocolim}{hocolim}
\DeclareMathOperator*{\holim}{holim}

\DeclareMathOperator{\Fil}{Fil}
\DeclareMathOperator{\Ind}{Ind}

\DeclareMathOperator*{\modmod}{/\mkern-6mu/}
\newcommand{\sumsum}{\Sigma\mkern-10.5mu\Sigma}
\DeclareMathOperator*{\cocoprod}{\amalg\mkern-12mu\amalg}
\DeclareMathOperator*{\ttimes}{\times\mkern-9mu\times}

\title{Filtered spaces, filtered objects}
\author{Tyler Lawson}

\begin{document}
\maketitle

\begin{abstract}
  We introduce a operation on categories enriched in filtered spaces, whose effect is to turn categories of $E_1$-pages into categories of $E_2$-pages. This allows us to give a homotopical versions of several results that were previously implemented using $E_2$-model structures or more sophisticated machinery in higher algebra.

  We find that we can recover the homotopy theory of spaces from this page-turning operation on the homotopy theory of CW-complexes and filtration-shifting maps, a version of the cellular approximation theorem. In the category of filtered spectra, we show that this implements the procedure on filtered spectra sending the homotopy exact couple to its associated derived couple. Finally, we recover Pstr\k{a}gowski's category of synthetic spectra from applying this page-turning operation to the category of filtered modules over a spectral version of the Rees ring.
\end{abstract}

\section{Introduction}

Filtrations have been, and continue to be, a core calculational tool in homotopy theory. Spectral sequences are often built from a tower of interlocking exact sequences associated to a filtered object; obstruction theory uses cellular filtrations and Postnikov filtrations to analyze objects and maps.

A key component of these techniques is a \emph{lookahead} technique. A spectral sequence has differentials that capture obstructions to existence and uniqueness as an element passes through the associated tower; obstruction theory passes from cellular cochains to cohomology classes by looking at maps on one skeleton which are the restriction of maps on the next. This lookahead progressively eliminates obstructions to existence and uniqueness, removing dead ends from consideration---represented in a spectral sequence by ``turning the page''. It often allows us to begin with a tractable calculation and progress towards a complete one. Most hard calculations often rely on our ability to lift data from the algebraic end to the original filtered object.

More recent developments have led to better tools in this area, taking categories of resolutions and constructing categories with the flavor of an $E_2$-page. Originally these came from the ``resolution model structures,'' or $E_2$-model structures, used by Dwyer--Kan--Stover \cite{dwyerkanstover-e2},  Bousfield \cite{bousfield-cosimplicial}, and Goerss--Hopkins \cite{goerss-hopkins-summary}, which incorporate (co)simplicial objects as their basic objects of study. More modern tools for stable theory came from Pstr\k{a}gowski's \emph{synthetic spectra} \cite{pstragowski-synthetic}, which has alternative approaches via Burklund--Hahn--Senger \cite{burklund-hahn-senger-artintate} or the Beilinson $t$-structure \cite{ariotta-chain}. These latter approaches are more robust (in particular, not requiring a model structure), but fundamentally made use of a large number of tools from stable homotopy theory.

Our goal in this paper is to introduce a new technique for producing ``$E_2$-categories'' like these, enjoying properties similar to those of the  nonabelian derived category. We believe that the construction here is of additional use for several reasons.
\begin{itemize}
\item It is widely applicable, requiring only that we have an extension of $\mc C$ to a category with \emph{filtered} spaces of maps. In particular, (co)completeness or presentability properties of $\mc C$ are useful, but not strictly necessary, to the construction.
\item It interacts well with filtration-compatible monoidal structures on $\mc C$.
\item It does not require stability or a $t$-structure, but in stable cases it is closely connected to d\'ecalage and the Beilinson $t$-structure on filtered objects.
\item It is readily iterable. For example, the category of filtered spectra can be thought of as a category of ``$E_1$-pages'', and we can inductively apply our construction to form categories of $E_2$-pages, $E_3$-pages, and so on.
\end{itemize}

In particular, this technique will allow us a better connection with techniques employed in algebra. When $R$ is a ring with an ideal $J$, there is an associated Rees ring; the theory of modules over the Rees ring interpolates between $R$-modules and modules over the associated graded $\gr(R)$ of the $J$-adic filtration. We will show that Pstr\k{a}gowski's synthetic spectra \cite{pstragowski-synthetic} can be interpreted in terms of modules over homotopy-theoretic version of the Rees ring.

\subsection{Summary}

The core idea goes back to the cellular approximation theorem.

A CW-complex $K$ has a nested sequence of skeleta
\[
  \sk^0 K \subset \sk^1 K \subset \sk^2 K \subset \dots
\]
forming a filtered space. Given CW-complexes $K$ and $L$, any map $K \to L$ is homotopic to a cellular map that sends $\sk^n K$ into $\sk^n L$. Two homotopic maps are not usually homotopic through cellular maps: however, cellular approximation says that a homotopy $K \times [0,1] \to L$ can be chosen that sends $\sk^n K \times [0,1]$ into $\sk^{n+1} L$.

If we introduce the notation $\sh(\sk L)$ for the \emph{shift} of this filtered space, given by $\sk^{n+1} L$ in degree $n$, then we have a natural inclusion $\sk L \to \sh(\sk L)$. The cellular approximation theorem then tells us:
\begin{itemize}
\item up to equivalence, any map $K \to L$ comes from a point from the space $\Map(\sk K,\sk L)$ of filtration-preserving maps;
\item up to equivalence, any homotopy between two such maps $K \to L$ comes from a path in the space $\Map(\sk K,\sh(\sk L))$; 
\item and this continues in a clear pattern for higher homotopies.
\end{itemize}

This leads us to consider the following construction. Given a filtered space $X^0 \to X^1 \to \dots$, we can produce a new space $X^\der$ as a simplicial set:
\begin{itemize}
\item the vertices of $X^\der$ are points of $X^0$;
\item the paths between two vertices of $X^\der$ are paths between their images in $X^1$;
\item the 2-simplices of $X^\der$ extend their boundary via a $2$-simplex in $X^2$;
\item and so on.
\end{itemize}
This new space is analyzable in terms of $X$. Moreover, it extends readily to being a filtered space itself, and so the construction can be applied inductively.

If $\mc C$ is a category enriched in filtered spaces, such as the category of CW-complexes, we can then apply this construction to each mapping space in $\mc C$ to produce a category $\mc C^\der$.

\subsection{Outline}

In Section~\ref{sec:successor} we introduce the main page-turning construction of this paper, taking a filtered space $X$ and producing a new \emph{successor} space $X^\der$, by a recipe similar to the above. The bulk of Section~\ref{sec:successor} is devoted to showing that this has several good properties: it has homotopy groups calculable in terms of $X$, and interacts well with many categorical constructions. The hardest result in this section is Theorem~\ref{thm:Dpullback}, which shows how $X \mapsto X^\der$ interacts with general homotopy pullbacks. To the reader who is not sure why we begin this paper by using point-set models such as cubical sets: in this model Theorem~\ref{thm:Dpullback} is nearly self-evident, whereas it is difficult to intractable in some others. The reader who is willing to take this result as given might skim through the more technical work in this section.

In Section~\ref{sec:homotopical} we give two alternative, model-independent descriptions of $X^\der$, one using mapping spaces and one in terms of a localization of filtered spaces. In particular, we show in Proposition~\ref{prop:beilinsonweight} that if $X = \Omega^\infty Y$ for a filtered spectrum $Y$, then $X^\der$ can be recovered from the \emph{d\'ecalage} construction on $Y$, part of the Beilinson $t$-structure on filtered spectra. We also show how it factors the geometric realization of simplicial spaces through filtered spaces.

In Section~\ref{sec:monoidality} we explain the sense in which $X \mapsto X^\der$ is monoidal---both for the point-set version and the homotopical version. The main theorem is Theorem~\ref{thm:Dmonoidal}, showing that the construction is coherently lax symmetric monoidal.

Once this is done, we can apply it to enriched categories. Section~\ref{sec:successorcat} takes a category $\mc C$ enriched in filtered spaces and associates to it a category $\mc C^\der$. We then demonstrate how homotopy limits and colimits in $\mc C^\der$ can be deduced from those in $\mc C$, as a consequence of Theorem~\ref{thm:Dpullback}. Of particular interest are categories of filtered objects.

In Section~\ref{sec:diagrams} we show how this structure interacts with diagrams of categories and adjunctions, and how monoidal structures on $\mc C$ can give rise to monoidal structures on $\mc C^\der$.

Section~\ref{sec:stable} analyzes the special case when $\mc C$ is \emph{stable}. In particular, $\mc C^\der$ is often stable, and Proposition~\ref{prop:cofiberpres} gives a criterion for when cofiber sequences in $\mc C$ can remain cofiber sequences in $\mc C^\der$ that is useful for applications.

The rest of the paper is devoted to examples. In Section~\ref{sec:cw} we  prove that we can recover the homotopy theory of CW-complexes by applying this construction to their filtered mapping spaces. Section~\ref{sec:filteredspectra} is about the category of filtered spectra, and shows that these constructions recover the process of sending the exact couple for the homotopy of a filtered spectrum to the associated derived couple. Section~\ref{sec:ssyn} shows that if we have a filtered ring, such as a Rees ring, we get an associated \emph{semisynthetic category} that interpolates between modules over the underlying ring and modules over the associated graded, and in Section~\ref{sec:syn} we show that Pstr\k{a}gowski's category of $E$-synthetic spectra is a category of semisynthetic modules over an Adam--Rees ring for $E$. Rees rings have long been employed in algebraic geometry as a tool for interpolating between a ring and a quotient; this gives us a more direct analogy between these deformation-theoretic techniques and the synthetic category. Further, we show that the homotopy category of ``$C\tau$-modules'' can be very directly connected with a category of chain-complex objects and chain homotopy classes of maps.

This paper is part of a long-term goal: to reimplement and extend the Goerss--Hopkins obstruction theory, including the technique of resolving an operad by simpler ones, into an obstruction theory that can more easily be applied to a wider variety of examples. We expect to return to this obstruction theory, and in particular to extended power constructions, in future work.

\subsection{Acknowledgements}

The author would like to thank
Clark Barwick,
Andrew Blumberg,
Jeremy Hahn,
Jacob Hegna,
Liam Keenan,
Haynes Miller,
and
Piotr Pstr\k{a}gowski
for discussions related to this paper. This work was partially supported by NSF grant DMS-2208062.

\section{Filtered spaces and their successors}
\label{sec:successor}
\subsection{Filtrations}

\begin{defn}
  \label{def:filtered}
  For a category $\mc C$, a \emph{filtered object} $X$ of $\mc C$ is a functor from the poset $\mb Z$ to $\mc C$, viewed as a sequence of maps
  \[
    \dots \to X^0 \to X^1 \to X^2 \to X^3 \to \dots
  \]
  We write $\Fil(\mc C)$ for the category of filtered objects.

  Similarly, a \emph{nonnegatively filtered object} is a functor from the poset $\mb N$ to $\mc C$, which form the category $\Fil^+(\mc C)$.
\end{defn}

We regard each filtered object as having an underlying nonnegatively filtered object. We can also regard nonnegatively filtered objects as equivalent to filtered objects which are constant in nonpositive degrees. (If $\mc C$ has an initial object $\emptyset$, we could instead regard nonnegatively filtered objects as equivalent to filtered objects isomorphic to $\emptyset$ in negative degrees.)

\begin{exam}
  \label{exam:skeletalfilt}
  For a CW-complex $W$, the sequence of skeleta $W^{(n)} \subset W$ determines a nonnegatively filtered topological space, which we refer to as the \emph{skeletal filtration} and write as $\sk(W)$. Similarly, for a simplicial or cubical set $K$, the sequence of skeleta $\sk^n K \subset K$ determines a natural nonnegatively filtered object also written as $\sk(K)$.
\end{exam}

\begin{defn}
  \label{def:shift}
  For a filtered object $X$ and an integer $k$, the \emph{shift of $X$ by $k$}, $\sh^k(X)$, is the composite of the functor $n \mapsto n+k\co \mb Z \to \mb Z$ with $X$. For nonnegatively filtered objects $X$, we similarly define $\sh^k(X)$ for $k \geq 0$. When $k=1$, we simply write $\sh(X)$. 
  
  Roughly,
  \[
    \sh(\dots \to X^0 \to X^1 \to X^2 \to \dots) = (\dots \to X^1 \to X^2 \to X^3 \to \dots).
  \]

  The maps $X^n \to X^{n+1}$ assemble into natural transformations $\tau^k\co X \to \sh^k X$ for any $k \geq 0$, satisfying $\sh(\tau) = \tau$.
\end{defn}

\begin{defn}
  \label{def:freeeval}
  For any $n \in \mb Z$ there are evaluation functors
  \[
    \ev_n\co \Fil(\mc C) \to \mc C,
  \]
  given by $X \mapsto X^n$. If $\mc C$ has an initial object $\emptyset$, then these have left adjoints
  \[
    F_n\co \mc C \to \Fil (\mc C),
  \]
  such that
  \[
    (F_n W)^k = \begin{cases}
      W &\text{if }k \geq n,\\
      \emptyset &\text{if }k < n,
    \end{cases}
  \]
  with most structure maps being identity maps.

  Similar adjoint pairs are defined for $\Fil^+(\mc C)$ for $n \in \mb N$.
\end{defn}

\subsection{The successor functor}

Functoriality of skeleta in cubical sets allows us to make the following definition.

\begin{defn}
  \label{def:Ddefinition}
  Suppose $X$ is a (nonnegatively) filtered cubical set. The associated \emph{successor cubical set} $X^\der$ is, in degree $n$, the set of maps of filtered objects
  \[
    \sk(\square^n) \to X.
  \]
  More precisely, $X^\der$ is the composite functor
  \[
    \Hom(\sk(-),X) \co \square^\op \to \Fil(\Set^\square)^\op \to \Set.
  \]
\end{defn}

In particular, for a filtered cubical set $X^0 \to X^1 \to X^2 \to \dots$, a point of $X^\der$ is the same as a point of $X^0$; an edge in $X^\der$ is a pair of points of $X^0$, together with an edge between their images in $X^1$; and so on.

The definition of the successor automatically makes the following result hold.
\begin{prop}
  \label{prop:Dadjunction}
  There is a natural isomorphism
  \[
    \Hom_{\Fil(\cSet)}(\sk(K), X) \cong \Hom_{\cSet}(K,X^\der),
  \]
  making the successor $(-)^\der$ into a functor right adjoint to $\sk$.
\end{prop}

\begin{prop}
  The successor functor $(-)^\der$ lifts to functors $\Fil(\cSet) \to \Fil(\cSet)$ and  $\Fil^+(\cSet) \to \Fil^+(\cSet)$.
\end{prop}

\begin{proof}
  If $X$ is a filtered cubical set, it has a functorially associated sequence
  \[
    \dots \to X \to \sh X \to \sh^2 X \to \dots
  \]
  of filtered cubical sets. Taking successors gives a new filtered cubical set
  \[
    \dots \to X^\der \to (\sh X)^\der \to (\sh^2 X)^\der \to \dots.
  \]
  The same argument applies to nonnegatively filtered cubical sets.
\end{proof}

\begin{rmk}
  \label{rmk:Dhocolim}
  If the maps $X^n \to X^{n+1}$ are isomorphisms for $n \geq 0$, making $X$ eseentially a constant filtered cubical set, then $X^\der \cong X^0$.

  If we view the maps $X^m \to \colim_{n \in \mb Z} X^n$ as determining a map from $X$ to a constant filtered object, taking successors then gives a natural transformation
  \[
    X^\der \to \colim_{n \in \mb Z} X^n \simeq \hocolim_{n \in \mb Z} X^n.
  \]
  Similarly, the maps $X^0 \to X^n$ determine a map $X^0 \to X^\der$.
\end{rmk}

\subsection{Fibration properties}

The following technical result tells us that $X^\der$ is often well-behaved homotopically.

\begin{defn}
  \label{def:defibrationcond}
  Suppose that $f\co X \to Y$ is a map of filtered spaces. We say that $f$ satisfies the \emph{defibration condition} if, for all $m \geq 0$, the map of relative homotopy sets
  \[
    \pi_m(X^m, X^{m-1}) \to \pi_m(Y^m, Y^{m-1})
  \]
  is surjective (at any basepoint).

  A map of filtered cubical sets or filtered simplicial sets satisfies the defibration condition if it does so on geometric realization.
\end{defn}

The defibration condition will allow us to convert fibrations to fibrations. However, we first require the following basic lemma about model categories.
\begin{lem}
  \label{lem:homotopylift}
  Suppose that we have a model category $\mc M$ with a diagram
  \[
    \begin{tikzcd}
      A \ar[r] \ar[d,hook] & X \ar[d,two heads]\\
      B \ar[r] & Y
    \end{tikzcd}
  \]
  such that
  \begin{enumerate}
  \item either $A$ is cofibrant or $\mc M$ is left proper, and
  \item either $Y$ is fibrant or $\mc M$ is right proper.
  \end{enumerate}
  Then there is a lift $B \to X$ in $\mc M$ if and only if there is a lift in the homotopy category $h\mc M$.
\end{lem}

\begin{proof}
  Consider the induced model structure on $\mc M_{A / -/ Y}$, the category of objects under $A$ and over $Y$. The object $B$ is cofibrant and $X$ is fibrant in $\mc M_{A / - / Y}$, and so every map in the homotopy category $h\mc M_{A / - / Y}$ comes from a genuine map in $\mc M_{A / - / Y}$. Thus, it suffices to show that the existence of a lift in $h \mc M$ implies the existence of a lift in $h\mc M_{A / - / Y}$.

  Choose a cofibrant replacement $A_c$ of $A$ and a fibrant replacement $Y_f$ of $Y$, then expand to a diagram:
  \[
    \begin{tikzcd}
      A_c \ar[r, two heads, "\sim"] \ar[d, hook]&
      A \ar[r] \ar[d,hook] &
      X \ar[r, hook, "\sim"] \ar[d, two heads] &
      X_f \ar[d, two heads] \\
      B_c \ar[r, two heads, "\sim"']&
      B \ar[r] &
      Y \ar[r, hook, "\sim"'] &
      Y_f
    \end{tikzcd}
  \]
  If there exists a lift in $h\mc M$, we can represent this map in $h\mc M$ by a dotted lift $B_c \to X_f$ in this diagram. Thus, we find that there is a map $B_c \to X_f$ in $\mc M_{A_c / - / Y_f}$, and hence a map $B \to X$ in the homotopy category $h\mc M_{A_c / - / Y_f}$.

  Now consider the composite functor
  \[
    \mc M_{A / - / Y} \to \mc M_{A_c / - / Y} \to \mc M_{A_c / - / Y_f},
  \]
  where the first functor is right Quillen and the second is left Quillen. The first is a Quillen equivalence if either $\mc M$ is left proper or $A_c \to A$ is a weak equivalence between cofibrant objects; the second is a Quillen equivalence if either $\mc M$ is right proper or $Y \to Y_f$ is a weak equivalence between fibrant objects.
\end{proof}

\begin{lem}
  \label{lem:Dfibration}
  Suppose that $f\co X \to Y$ is a map of filtered cubical sets that is a levelwise fibration. If $f$ satisfies the defibration condition, then the map $X^\der \to Y^\der$ is a fibration of cubical sets.
\end{lem}

\begin{proof}
  We need to show that the map $X^\der \to Y^\der$ satisfies the right lifting property against any cap inclusion $\sqcap^m_{k,\epsilon} \to \square^m$. By the adjunction of Proposition~\ref{prop:Dadjunction}, this is equivalent to showing that the map $X \to Y$ satisfies the right lifting property against the map $\sk(\sqcap^m_{k,\epsilon}) \to \sk(\square^m)$ of filtered cubical sets.

  The maps $sk^n \sqcap^m_{k,\epsilon} \to \sk^n \square^m$ are isomorphisms for $n < m-1$, and so there exists a unique lift below filtration $m-1$. Moreover, the maps $\sk^n \sqcap^m_{k,\epsilon} \to \sk^{n+1} \sqcap^m_{k,\epsilon}$ and $\sk^n \square^m \to \sk^n \square^m$ are isomorphisms for $n \geq m$, and so if a lift is constructed up to degree $m$, it automatically extends to a lift of filtered objects. To complete the proof, then, it suffices to take the commutative cube of maps in degrees $m-1$ and $m$,
  \[
    \begin{tikzcd}[row sep=small, column sep=small]
      \sqcap^m_{k,\epsilon} \ar[rr] \ar[dr] \ar[dd] &&
      X^{m-1} \ar[dr] \ar[dd] \\
      & \sqcap^m_{k,\epsilon} \ar[rr, crossing over] &&
      X^m \ar[dd] \\
      \partial \square^m \ar[dr] \ar[rr] &&
      Y^{m-1} \ar[dr] \\
      & \square^m \ar[rr] \ar[from=uu, crossing over] &&
      Y^m,
    \end{tikzcd}
  \]
  and construct lifts $\partial \square^m \to X^{m-1}$ and $\square^m \to X^m$ making the whole diagram commute. Alternatively, in the \emph{arrow category} $\Fun(\{0 \to 1\}, \cSet)$, it suffices to show that there is a lift in the diagram
  \[
    \begin{tikzcd}
      (\sqcap^m_{k,\epsilon}, \sqcap^m_{k,\epsilon}) \ar[r] \ar[d] &
      (X^{m-1}, X^m) \ar[d] \\
      (\square^m, \partial \square^m) \ar[r] \ar[ur, dotted]&
      (Y^{m-1}, Y^m).
    \end{tikzcd}
  \]
  (When $m > 0$, such lifts are the same as pointed lifts.)
  In the projective model structure on the arrow category of pointed cubical sets, the left-hand map is a cofibration between cofibrant objects and the right-hand map is a fibration; moreover, the category of cubical sets is right proper by \cite[8.4.38]{cisinski-testcat}, and the arrow category inherits this. Therefore, by Lemma~\ref{lem:homotopylift}, a lift exists if and only if one exists in the homotopy category of the arrow category. In this homotopy category, however, the map $(\sqcap^m_{k,\epsilon}, \sqcap^m_{k,\epsilon}) \to (\square^m, \partial \square^m)$ is isomorphic to the map $(\ast, \ast) \to (D^m, S^{m-1})$, and so this is precisely asking that an element in $\pi_m(Y^m, Y^{m-1})$ lifts to an element in $\pi_m(X^m, X^{m-1})$, as desired.
\end{proof}

Applying this to a map $X \to \ast$, we obtain the following result.
\begin{cor}
  \label{cor:Dfibrancy}
  If $X$ is a (levelwise) fibrant filtered cubical set, then $X^\der$ is fibrant. In particular, this is true if $X$ is the singular cubical set of a filtered topological space.
\end{cor}

\subsection{Homotopy groups}

The explicit description of the successor functor allows us to analyze its effect on homotopy groups.

\begin{thm}
  \label{thm:Dhomotopy}
  Suppose that $X$ is a fibrant filtered cubical set. Then the map $X^0 \to X^\der$ induces a natural isomorphism
  \[
    \pi_0(X^\der) \cong \im(\pi_0 X^0 \to \pi_0 X^1).
  \]
  Similarly, given any basepoint $x \in X^0$, with associated image basepoints in $X^m$, there are natural isomorphisms
  \[
    \pi_m(X^\der,x) \cong \im(\pi_m(X^m, x) \to \pi_m(X^{m+1},x))
  \]
  of groups.
\end{thm}

\begin{proof}
  For convenience, we will leave the basepoints implicit in this proof.

  This follows from an explicit description of the homotopy groups of a fibrant cubical set, using models for $S^n$ and $S^n \times [0,1]$. Namely, each element in $\pi_m X^\der$ is represented by a map $\alpha\co \square^m \to X^\der$ taking the boundary to the basepoint, while two elements $\alpha$ and $\beta$ are equivalent if there is a homotopy $H\co \square^{m+1} \to X^\der$ taking the first face to $\alpha$, its opposing face to $\beta$, and all other faces to the basepoint.

  By adjunction, these correspond to maps of filtered cubical sets. An element $\alpha$ corresponds to a map $\sk(\square^m) \to X$ taking $\sk(\partial \square^m)$ to the basepoint; this is the same as a map $\square^m \to X^m$ taking the boundary to the basepoint, and thus is a representative for an element $\alpha \in \pi_m X^m$. A homotopy $H$ between two elements $\alpha$ and $\beta$ corresponds to a map $\sk(\square^{m+1}) \to X$, and the face properies of $H$ translate into a homotopy in $X^{m+1}$ between the images of $\alpha$ and $\beta$.

  This constructs a well-defined natural bijection
  \[
    \pi_m(X^\der) \cong \im(\pi_m(X^m) \to \pi_m(X^{m+1})).
  \]
  Similarly, for $m \geq 1$ we can verify that this natural isomorphism respects the group structure from an explicit description of the group operation. We have an identity $\alpha \cdot \beta = \gamma$ if an only if there is a map $\square^{m+1} \to X^\der$ taking three appropriately oriented faces to $\alpha$, $\beta$, and $\gamma$, while the remaining faces are the basepoint. Under the adjunction, this is equivalent to saying that $\alpha \cdot \beta = \gamma$ in $\pi_m(X^{m+1})$.
\end{proof}

\begin{rmk}
  Theorem~\ref{thm:Dhomotopy} can be used to show that the successor functor extends to a well-defined derived functor from filtered spaces to spaces. However, we will give a different proof in a later section.
\end{rmk}

\subsection{Categorical products}

The fact that the successor functor is a right adjoint automatically implies that it preserves limits. In most cases it does not preserve homotopy limits. In the special case of products, however, we record the following.

\begin{prop}
  \label{prop:Dproducts}
  The successor functor preserves products: $(\prod X_i)^\der \to \prod X^\der_i$ is always an isomorphism. In particular, if the $X_i$ are fibrant, then so is $\prod X_i$, and so the successor functor preserves homotopy products.
\end{prop}

\begin{proof}
  Theorem~\ref{thm:Dhomotopy} implies that the natural map on homotopy groups is an isomorphism at any basepoint.
\end{proof}

\subsection{Homotopy pullbacks}

The defibration condition is one condition to guarantee that homotopy pullbacks are preserved by successors.
\begin{prop}
  \label{prop:Ddefibrationpullback}
  If $X \to Z \leftarrow Y$ is a diagram of fibrant filtered cubical sets and $Y \to Z$ satisfies the defibration condition, then $(X \times_Z Y)^\der$ is a representative for the homotopy pullback of the diagram $X^\der \to Z^\der \leftarrow Y^\der$.
\end{prop}

\begin{proof}
  Because $(-)^\der$ is a right adjoint, it preserves pullbacks, and by Lemma~\ref{lem:Dfibration} it preserves fibrations that satisfy the defibration condition.
\end{proof}

The successor functor does not usually preserve homotopy pullbacks. However, we can analyze exactly how it deviates.

Cubical sets have a nonsymmetric monoidal product, $\boxtimes$, such that $\square^n \boxtimes \square^m \cong \square^{n+m}$. This preserves colimits in each variable and satisfies the pushout-product axiom. Correspondingly, there is a functor $\Map^r$ with a natural isomorphism
\[
  \Hom_{\cSet}(K \boxtimes L, M) \cong \Hom_{\cSet}(K, \Map^r(L,M))
\]
satisfying the analogue of the SM7 axiom.

The functor $\Map^r$ allows us to construct a homotopy pullback functor. For a diagram $K \to M \leftarrow L$ of cubical sets, there is a fiber product
\[
  K \times^h_M L = (K \times L) \times_{M \times M} \Map^r(\square^1, M),
\]
which is a homotopy pullback if all three cubical sets are fibrant.

In particular, note that a point of $X^\der \times^h_{Z^\der} Y^\der$ consists of points of $X^0$ and $Y^0$, together with a path between \emph{their images in $Z^1$}---which we can view as a path in the shift of $Z$, rather than a path in $Z$. Our goal in Theorem~\ref{thm:Dpullback} is to generalize this, proving that $(X \times^h_{\sh Z} Y)^\der$ serves as a homotopy pullback of $X^\der \to Z^\der \leftarrow Y^\der$.

\begin{thm}
  \label{thm:Dpullback}
  Associated to a diagram $X \to Z \leftarrow Y$ of filtered cubical sets, there is a natural \emph{isomorphism}
  \[
    X^\der \times^h_{Z^\der} Y^\der \to (X \times^h_{\sh Z} Y)^\der.
  \]
\end{thm}

\begin{proof}
  A map of cubical sets $K \to X^\der \times^h_{Z^\der} Y^\der$ consists of compatible maps $K \to X^\der, K \to Y^\der$, and $K \boxtimes \square^1 \to Z^\der$. By adjunction, this is equivalent to a commutative diagram
  \begin{equation}
    \label{eq:preskeleta}
    \begin{tikzcd}
      \sk(K) \times \{0\} \ar[r] \ar[d] &
      \sk(K \boxtimes \square^1) \ar[d] &
      \sk(K) \times \{1\} \ar[l] \ar[d] \\
      X \ar[r] &
      Z &
      Y. \ar[l]
    \end{tikzcd}
  \end{equation}
  The skeleton functor for cubical sets interacts well with products: there is a natural isomorphism
  \[
    \sk^n(K \boxtimes L) \cong \bigcup_{p+q=n} \sk^p(K) \boxtimes \sk^q(L).
  \]
  Because $\square^1$ is 1-dimensional, we find that a map $\sk(K \boxtimes \square^1) \to Z$ is equivalent to a diagram of maps $\sk^n(K) \boxtimes \partial \square^1 \to Z^n$ for the boundary, together with a diagram of maps $\sk^n(K) \boxtimes \square^1 \to Z^{n+1}$, such that the two resulting maps $\sk^n(K) \boxtimes \partial \square^1 \to Z^{n+1}$ agree. We can re-express these maps as maps into the shift of $Z$, and use the adjunction with $\Map^r$. This shows that a map $\sk(K \boxtimes \square^1) \to Z$ is essentially a commutative diagram 
  \[
    \begin{tikzcd}
      \sk(K) \ar[r] \ar[d] &
      Z \times Z \ar[d] \\
      \Map^r(\square^1, \sh Z) \ar[r] &
      \sh Z \times \sh Z.
    \end{tikzcd}
  \]

  Combining this with equation~\eqref{eq:preskeleta}, we find that a map $\sk(K) \to X^\der \times^h_{Z^\der} Y^\der$ is naturally isomorphic data to a commutative diagram
  \[
    \begin{tikzcd}
      \sk(K) \ar[r] \ar[d] &
      X \times Y \ar[d] \\
      \Map^r(\square^1, \sh Z) \ar[r] &
      \sh Z \times \sh Z.
    \end{tikzcd}
  \]
  However, this is precisely a map $\sk(K) \to X \times^h_{\sh Z} Y$, or a map $K \to (X \times^h_{\sh Z} Y)^\der$. Since this isomorphism of mapping sets is natural in $K$ and the pullback diagram, we find that there is an associated natural isomorphism $X^\der \times^h_{Z^\der} Y^\der \cong (X \times^h_{\sh Z} Y)^\der$ as desired.
\end{proof}

\begin{cor}
  \label{cor:Dhpullback}
  If $X \to Z \leftarrow Y$ is a diagram of fibrant filtered cubical sets, then $(X \times^h_{\sh Z} Y)^\der$ is a representative for the homotopy pullback of the diagram $X^\der \to Z^\der \leftarrow Y^\der$.
\end{cor}

\subsection{Filtered hocolims}

\begin{prop}
  \label{prop:Dfilteredhocolim}
  The successor functor preserves filtered homotopy colimits.
\end{prop}

\begin{proof}
  The natural map (at any basepoint)
  \[
    \colim \pi_n(X_i) \to \pi_n(\hocolim_i X_i)
  \]
  is an isomorphism if the diagram is filtered, and filtered colimits preserve images. Therefore, for such a diagram $X_i$ of fibrant filtered cubical sets, the natural functor $\hocolim_I X^\der_i \to (\hocolim_i X_i)^\der$ is an isomorphism on $\pi_n$ (at any basepoint) by Theorem~\ref{thm:Dhomotopy}, and hence is an equivalence.
\end{proof}

\subsection{Filtered holims}

It is best to be realistic in hopes that successors could preserve a filtered homotopy limit.

\begin{exam}
  \label{ex:Dholimfailure}
  For any pointed space $Z$ and any $n \geq 0$, we can construct a filtered space
  \[
    \dots \to \ast \to \ast \to Z \xrightarrow{=} Z \to \ast \to \dots,
  \]
  given by $Z$ in degrees $n$ and $(n+1)$ and trivial otherwise. After taking successors, Theorem~\ref{thm:Dhomotopy} shows that we get an Eilenberg--Mac Lane space $K(\pi_n Z, n)$.

  Now fix such an $n$ and consider any tower $\dots \to Z_2 \to Z_1 \to Z_0$ of pointed spaces. If we applying the above construction, we get a tower of filtered spaces.

  If we first take homotopy limits and then take successors, we get an Eilenberg--Mac Lane space $K(\pi_n \holim Z_i, n)$.

  If we first take successors, we get a tower of Eilenberg--Mac Lane spaces $K(\pi_n Z_i, n)$. Then taking homotopy limits, Milnor's $\lim^1$-sequence implies that we get a space whose homotopy groups are $\lim \pi_n Z_i$ in degree $n$ and $\lim^1 \pi_n Z_i$ in degree $(n-1)$.
\end{exam}

\subsection{Nearly constant spaces}

\begin{defn}
  \label{def:nearlyconstantspace}
  We say that an object $X \in \Fil^+(\mc C)$ is \emph{nearly constant} if there is a nonnegatively filtered object $Z$ and a diagram
  \[
    \begin{tikzcd}
      Z \ar[rr,"\tau"] \ar[dr] && \sh Z\\
      & X \ar[ur]
    \end{tikzcd}
  \]
  such that $\tau\co Z \to \sh Z$ is an equivalence. We refer to $Z^0$ as the \emph{value} of the nearly constant object.
\end{defn}

\begin{rmk}
  This is equivalent to a collection of objects $Z^n \in \mc C$ together with factorizations
  \[
    Z^0 \to X^0 \to Z^1 \to X^1 \to Z^2 \to X^2 \to \dots
  \]
  such that the double-composites $Z^n \to X^n \to Z^{n+1}$ are equivalences.
\end{rmk}

\begin{prop}
  \label{prop:Dnearlyconstant}
  Suppose that $X$ is a fibrant (nonnegatively) filtered cubical set which is nearly constant with value $Z^0$. Then the composite map $Z^0 \to X^0 \to X^\der$ is an equivalence.
\end{prop}

\begin{proof}
  Choose a sequence
  \[
    Z^0 \to X^0 \to Z^1 \to X^1 \to Z^2 \to X^2 \to \dots
  \]
  with the maps $Z^n \to Z^{n+1}$ being equivalences. For any $n$ (and at any basepoint if $n > 0$), we get a factorization:
  \[
    \begin{tikzcd}
      \pi_n Z^n \ar[rr,"\sim"] \ar[dr] && \pi_n Z^{n+1} \ar[dr] \ar[rr,"\sim"] && \pi_n Z^{n+2}\\
      & \pi_n X^n \ar[rr] \ar[ur] && \pi_n X^{n+1} \ar[ur]
    \end{tikzcd}
  \]
  A straightforward check then shows that map $\pi_n Z^n \to \im(\pi_n X^n \to \pi_n X^{n+1})$ is an isomorphism (at any basepoint if $n > 0$). Therefore, the composite $Z^0 \to X^\der \to X^\der$ is an equivalence.
\end{proof}

\begin{cor}
  \label{cor:nearlycontractible}
  If $X$ is a fibrant (nonnegatively) filtered cubical set such that $X^0$ is nonempty and the maps $X^n \to X^{n+1}$ are homotopic to constant maps, then $X^\der$ is contractible.
\end{cor}

\begin{proof}
  Any point of $X^0$, and chosen homotopies of the maps to constant maps, explicitly determine a sequence
  \[
    \ast \to X^0 \to CX^0 \to X^1 \to CX^1 \to \dots
  \]
  where $CX^n$ is the cone. Therefore, $X$ is nearly constant with value $\ast$.
\end{proof}

\section{Model-independent description}
\label{sec:homotopical}

It will be convenient to have an alternative descriptions of $X^\der$ which are manifestly homotopy-invariant; we will give two. The first is not point-set functorial, but is purely homotopy-theoretic, and is built from the same ideas as the Beilinson $t$-structure. The second shows that in the construction of $X^\der$ we may either take function sets or function spaces.

\subsection{Weight for filtered spaces}

\begin{defn}
  \label{def:filteredinfty}
  Write $\Spaces$ for the $\infty$-category of spaces, and $\Fil(\Spaces)$ for the functor $\infty$-category $\Fun(\mb Z, \Spaces)$ of filtered spaces.
\end{defn}

\begin{defn}
  \label{def:nonnegative}
  A filtered space $X$ is \emph{weight at least $w$} if, for all $m$, the map $X^m \to X^{m+1}$ is $(m+w)$-connected. We use
  \[
    \Fil(\Spaces)^{\geq w} \subset \Fil(\Spaces)
  \]
  to denote the full subcategory of filtered spaces that are of weight at least $w$. (We sometimes also say that $X \geq w$.)

  In particular, a filtered space has \emph{nonnegative weight} if it is weight at least $0$.
\end{defn}

\begin{exam}
  \label{ex:sspace-skel}
  If $Y\co \Delta^\op \to \Spaces$ is a simplicial space with homotopy colimit $|Y|$, then the skeletal filtration produces a filtered space $\sk |Y|$ with nonnegative weight.
\end{exam}

\begin{rmk}
  These definitions automatically extend to nonnegatively filtered spaces.
\end{rmk}

\begin{rmk}
  \label{rmk:relativeCW}
  Given a filtered space of weight at least $w$ represented by a sequence $\dots \to X^0 \to X^1 \to X^2 \to \dots$, we can apply a relative version of CW approximation. We find that, up to equivalence, we may assume that each $X^m$ is a CW-complex, that each map $X^m \to X^{m+1}$ is formed by attaching cells of dimension at least $w+m+1$, and that the intersection of the $X^m$ is empty.

  Alternatively, we could interpret such an $X$ as a CW-complex with an integer degree assigned to each cell, with the degree of a cell being at least as large as the degree of any cell on its boundary, that satisfies $\mathrm{degree}(e) \leq \mathrm{dim}(e) - w$.
\end{rmk}

\begin{rmk}
  Note that $X$ is weight at least $w$ if and only if $\sh X$ is weight at least $w+1$. (The forward implication is still true for nonnegatively filtered spaces.) In particular, statements about an individual weight can be typically be reduced to objects of weight zero.
\end{rmk}

\begin{prop}
  \label{prop:Dnonneg}
  If $X$ is a fibrant filtered cubical set that represents a filtered space of weight at least $0$, the map $X^\der \to \hocolim_{\mb Z} X$ of Remark~\ref{rmk:Dhocolim} is an equivalence.
\end{prop}

\begin{proof}
  This follows immediately from the calculation of the homotopy groups of $X^\der$ in Theorem~\ref{thm:Dhomotopy}.
\end{proof}

\subsection{Categorical properties of weight}
\begin{prop}
  \label{prop:nonneghocolim}
  For any $w$, the subcategory $\Fil(\Spaces)^{\geq w} \subset \Fil(\Spaces)$ is closed under homotopy colimits.
\end{prop}

\begin{proof}
  In a functor category, homotopy colimits are calculated pointwise \cite[5.1.2.3]{lurie-htt}. Therefore, this reduces to the fact that for a diagram $X_i \to Y_i$ of $m$-connected maps of spaces, the induced map $\hocolim_I X_i \to \hocolim_I Y_i$ is also $m$-connected.

  This can be proven, for example, by using an obstruction-theoretic criterion. A map $X \to Y$ of spaces is $m$-connected if and only if, for any map $Z \to W$ such that $\pi_n(W,Z)$ identically vanishes for $n > m$, the map
  \[
    \Map(Y,Z) \to \Map(X,Z) \times^h_{\Map(X,W)} \Map(Y,W)
  \]
  is a weak equivalence. This criterion is closed under homotopy colimits: both sides take homotopy colimits in $X \to Y$ to homotopy limits.
\end{proof}

\begin{thm}
  \label{thm:nonnegpresentable}
  The $\infty$-category $\Fil(\Spaces)^{\geq w}$ is compactly generated. In particular, it is presentable.
\end{thm}

\begin{proof}
  Let $\mc J \subset \Fil(\Spaces)^{\geq w}$ be the full subcategory spanned by filtrations of a finite CW-complex by subcomplexes. Namely, objects of $\mc J$ are represented by sequences of spaces $\dots \to X^0 \to X^1 \to X^2 \to \dots$ with the following properties:
  \begin{itemize}
  \item $X^m$ is empty for $m$ sufficiently small;
  \item each map $X^m \to X^{m+1}$ is a relative CW-inclusion formed by attaching finitely many cells, each of dimension greater than $w+m$; and
  \item the maps $X^m \to X^{m+1}$ are isomorphisms for sufficiently large $m$.
  \end{itemize}
  Such a filtration is always obtained from a finite CW-complex $X$ with a degree assigned to each cell, such that the degree on the boundary of a cell is no larger than the degree of the cell.
  
  By an inductive argument on the number of cells, we find that the objects of $\mc J$ are compact in $\Fil(\Spaces)$. Moreover, by Remark~\ref{rmk:relativeCW}, every object in $\Fil(\Spaces)^{\geq w}$ is represented by a certain type of sequence of relative CW-inclusions, and these are formally filtered colimits of their finite filtered CW-subcomplexes. This shows that, as $\infty$-categories, $\Fil(\Spaces)^{\geq w} \simeq \Ind(\mc J)$.
\end{proof}

\subsection{Truncation and weight}

\begin{prop}
  \label{prop:nonnegadjoint}
  The forgetful functor $\Fil(\Spaces)^{\geq w} \to \Fil(\Spaces)$ has a left adjoint $\tau_{\geq w}$. The map $\tau_{\geq w} X \to X$ is characterized by knowing that
  \begin{itemize}
  \item the space $\tau_{\geq w} X$ is of weight at least $w$, and
  \item for any $m$ and at any basepoint, the maps $(\tau_{\geq w} X)^m \to X^m$ are isomorphisms on $\pi_k$ for $k \geq m+w$ and injections on $\pi_{m+w-1}$.
  \end{itemize}

\end{prop}

\begin{proof}
  Both $\infty$-categories are presentable, and the inclusion functor preserves homotopy colimits, so the existence of the adjoint follows from the adjoint functor theorem \cite[5.5.2.9]{lurie-htt}. Moreover, the value of the adjoint functor on $X$ is homotopy terminal among filtered spaces of weight at least $w$ with a map to $X$.

  Fix a filtered space $X$. We claim that there is a filtered space $Y$, with a map $Y \to X$, satisfying the properties in the proposition. More specifically, we first set $Y^m = X^m$ for $m < -w$. Then, inductively, we factor the map $Y^m \to X^m \to X^{m+1}$ through a map $Y^m \to Y^{m+1} \to X^{m+1}$ by attaching cells of dimension greater than $m+w$ to correct the higher homotopy groups---a relative Postnikov factorization.

  We now claim that $Y$ is equivalent to the value of $\tau_{\geq w} X$. Given any $Z$ of weight at least $w$ with a map $Z \to X$, represent $Z$ by a sequence of relative CW-inclusions $Z^m \to Z^{m+1}$ formed by attaching cells of dimension greater than $m+w$. Then the restriction $\Map(Z,Y) \to \Map(Z,X)$ is an equivalence by obstruction theory, and hence $Y$ is homotopy terminal among objects of weight at least $W$ mapping to $X$.
\end{proof}

This description makes it clear that these functors interact well with the shift operator.
\begin{cor}
  \label{cor:weightadjoint}
  For a filtered space $X$, there is an equivalence $\tau_{\geq w+1} \sh X \to \sh \tau_{\geq w} X$.
\end{cor}

\begin{proof}
  The map $\tau_{\geq w+1} \sh X \to \sh X$ becomes a map $\sh^{-1}(\tau_{\geq w+1} \sh X) \to X$ which exhibits the source as a representative for $\tau_{\geq w} X$.
\end{proof}

\subsection{Successors and weight}
\begin{thm}
  \label{thm:Dhomotopical}
  Given a filtered space $X$ and a representative of the map $\tau_{\geq 0} X \to X$ by a map of fibrant filtered cubical sets, the induced map $(\tau_{\geq 0} X)^\der \to X^\der$ is an equivalence.
\end{thm}

\begin{proof}
  This follows directly from Proposition~\ref{prop:nonnegadjoint}, using the natural identification of the homotopy groups of $Y^\der$ in Theorem~\ref{thm:Dhomotopy}.
\end{proof}

Combining this with Proposition~\ref{prop:Dnonneg}, we find the following.
\begin{cor}
  \label{cor:Dhomotopical}
  Given a filtered space $X$ represented by a fibrant filtered cubical set, the space $X^\der$ is naturally equivalent to $\hocolim_{\mb Z} \tau_{\geq 0} X$.

  More generally, the sequence of maps
  \[
    \dots \to X \to \sh X \to \sh^2 X \to \dots
  \]
  becomes, upon taking successors, naturally equivalent to the sequence of maps
  \[
    \dots \to \hocolim_{\mb Z} \tau_{\geq 0} X \to \hocolim_{\mb Z} \tau_{\geq (-1)} X \to  \hocolim_{\mb Z} \tau_{\geq (-2)} X \to \dots
  \]
\end{cor}

\begin{exam}
  \label{ex:Dsspace}
  Given a simplicial space $Y$ with the associated filtration $\sk|Y|$ on its geometric realization as in Example~\ref{ex:sspace-skel}, we find that there is an equivalence
  \[
    (\sk|Y|)^\der \simeq |Y|.
  \]
  In particular, this gives a factorization of the geometric realization functor through filtered spaces:
  \[
    \begin{tikzcd}
      \Spaces^{\Delta^\op} \ar[dr, "\hocolim"']  \ar[r, "\sk"] &
      \Fil(\Spaces) \ar[d,"\der"]\\
      & \Spaces
    \end{tikzcd}
  \]
\end{exam}

\subsection{Relation to the Beilinson $t$-structure}

Our goal in this section is to show that the weight truncations $\tau_{\geq w}$ on filtered spaces are compatible with the truncation functors from the Beilinson $t$-structure.

\begin{prop}
  \label{prop:beilinsonweight}
  If we have a filtered spectrum $Y$ such that $Y \geq w$ in the Beilinson $t$-structure, then the filtered space $\Omega_\infty Y$ is of weight at least $w$.

  More generally, we have an equivalence
  \[
    \Omega^\infty \tau^B_{\geq w} Y \to \tau_{\geq w} \Omega^\infty Y
  \]
  where $\tau^B_{\geq w}$ denotes truncation in the Beilinson $t$-structure on filtered spectra.
\end{prop}

\begin{proof}
  Consider a filtered spectrum $Y$ represented by $\dots \to Y^0 \to Y^1 \to \dots$, with associated cofiber sequences
  \[
    Y^m \to Y^{m+1} \to \gr^{m+1} Y.
  \]
  We have $Y \geq w$ in the Beilinson $t$-structure if and only if in the associated graded $\gr(Y)$, we have $\gr^m Y \geq m+w$ in the standard $t$-structure on spectra. However, $\gr^m Y \geq m+w$ for all $m$ if and only if the maps $Y^m \to Y^{m+1}$ are all $(m+w)$-connected, and this connectivity is preserved by taking $\Omega^\infty$.

  Similarly, $Z < w$ in the Beilinson $t$-structure if and only if we have $Z^m < m+w$ in the standard $t$-structure on spectra. For a general filtered spectrum $Y$, there is a cofiber sequence $\tau^B_{\geq w} Y \to Y \to \tau^B_{< w}Y$ in the Beilinson $t$-structure. Term-by-term, we have cofiber sequences
  \[
    (\tau^B_{\geq w} Y)^m \to Y^m \to (\tau^B_{< w} Y)^m
  \]
  where the third term has trivial homotopy in degrees $(m+w)$ and above. Therefore, the map $(\tau^B_{\geq w} Y)^m \to Y^m$ is an isomorphism on $\pi_k$ for $k \geq m+w$ and an injection on $\pi_{m+w-1}$. This is preserved by applying $\Omega^\infty$. By Proposition~\ref{prop:nonnegadjoint}, then, $\Omega^\infty \tau^B_{\geq w} Y \to \Omega^\infty Y$ is a representative for $\tau_{\geq w} \Omega^\infty Y$.
\end{proof}

\begin{rmk}
  Recall that the generalized Dold--Kan correspondence of \cite[1.2.3]{lurie-higheralgebra} establishes an equivalence between nonnegatively filtered spectra and simplicial spectra. Under this correspondence, the Beilinson connective cover $\tau_{\geq w}^B$ is equivalent to the levelwise connective cover $\tau_{\geq w}$ on simplicial objects.

  For a simplicial spectrum $Y$, this and Example~\ref{ex:Dsspace} imply that there is a natural equivalence
  \[
    |\Omega^\infty Y| \simeq (\Omega^\infty \sk |Y|)^\der.
  \]
  In this sense, the successor functor allows the interchange of geometric realization with $\Omega^\infty$.
\end{rmk}

\subsection{Function sets and function spaces}

In this section, we will first outline how the successor functor can be equivalently described using simplicial sets rather than cubical ones, and then show that it can also use \emph{spaces} of functions rather than sets.

\begin{prop}
  For a levelwise fibrant filtered simplicial set, we have a natural weak equivalence
  \[
    X^\der \simeq \Hom_{\Fil(\sSet)}(\sk(\Delta^\bullet), X)
  \]
  of spaces.
\end{prop}

\begin{proof}[Proof sketch.]
  Just as in the cubical case, the assignment
  \[
    X \mapsto X^{\der'} = \Hom_{\Fil(\sSet)}(\sk(\Delta^\bullet), X)
  \]
  takes fibrant filtered simplicial sets to fibrant simplicial sets, takes constant objects to themselves, and has a natural isomorphism $\pi_n X^{\der'}  \cong \im(\pi_n X^n \to \pi_n X^{n+1})$. We therefore also get equivalences
  \[
    X^{\der'} \leftarrow (\tau_{\geq 0} X)^{\der'} \to \colim_{\mb Z} \tau_{\geq 0} X.
  \]
  This implies that $X^{\der'} \simeq X^\der$.
\end{proof}

We now prove that we can equally well use function spaces rather than function sets in the definition of $X^\der$.

\begin{prop}
  For a levelwise fibrant filtered simplicial set, we have a natural weak equivalence
  \[
    |\Hom_{\Fil(\sSet)}(\sk(\Delta^\bullet), X)| \to
    |\Map_{\Fil(\sSet)}(\sk(\Delta^\bullet), X)|
  \]
  of spaces.
\end{prop}

\begin{proof}
  We have a bisimplicial set which in degree $(p,q)$ is
  \begin{align*}
    \Hom_{\Fil(\sSet)}(\sk(\Delta^p) \times \Delta^q, X)
    &\cong \Hom_{\sSet}(\Delta^q, \Map_{\Fil(\sSet)}(\sk(\Delta^p), X))\\
    &\cong \Hom_{\Fil(\sSet)}(\sk(\Delta^p), \Map_{\sSet}(\Delta^q, X)).
  \end{align*}
  The edge with $q=0$ is the simplicial set $\Hom_{\Fil(\sSet)}(\sk(\Delta^p), X)$ whose realization models $X^\der$, and the geometric realizations in either order are isomorphic.

  If we take geometric realization in $q$ first, we obtain the simplicial space
  \[
    |\Hom_{\sSet}(\Delta^\bullet, \Map_{\Fil(\sSet)}(\sk(\Delta^p), X)| \cong \Map_{\Fil(\sSet)}(\sk(\Delta^p), X).
  \]
  Therefore, the full geometric realization is isomorphic to the realization of the simplicial space $|\Map_{\Fil(\sSet)}(\sk(\Delta^\bullet), X)|.$

  In the other direction, first taking realizations in $p$ we get
  \[
    |\Hom_{\Fil(\sSet)}(\sk(\Delta^\bullet), \Map_{\sSet}(\Delta^q, X))| = \Map_{\sSet}(\Delta^q, X)^\der.
  \]
  The edge inclusion maps $X \to \Map_{\sSet}(\Delta^q, X)$ are equivalences of filtered simplicial sets by levelwise fibrancy of $X$, and so the natural maps
  \[
    X^\der \to \Map_{\sSet}(\Delta^q,X)^\der
  \]
  are also equivalences because they are levelwise equivalences between fibrant filtered objects. Taking realization in $q$ we get an equivalence
  \[
    X^\der \xrightarrow{\sim} |\Map_{\Fil(\sSet)}(\sk(\Delta^\bullet), X)|,
  \]
  as desired.
\end{proof}

\begin{rmk}
  For a filtered space $X$, this gives an alternative model
  \[
    |\Map_{\Fil(\Spaces)}(\sk(\Delta^\bullet), X)|
  \]
  for $X^\der$.
\end{rmk}

\begin{rmk}
  The previous proposition also goes through in the cubical case, modulo taking some care with left and right mapping spaces and finding references for the realization of bicubical objects.
\end{rmk}

\section{Monoidality}
\label{sec:monoidality}

If $\mc C$ is a (symmetric) monoidal category with compatible colimits, the categories $\Fil(\mc C)$ and $\Fil^+(\mc C)$ of filtered objects of $\mc C$ are (symmetric) monoidal under the Day convolution. In particular, this is true for the category of filtered spaces. In this section we would like to examine the interaction of successors with this monoidal structure.

\subsection{Point-set monoidality of successors}

\begin{prop}
  \label{prop:Dlax}
  The successor functor $(-)^\der \co \Fil(\cSet) \to \cSet$ is lax monoidal for the Day convolution on the source.
\end{prop}

\begin{proof}
  The successor functor is right adjoint to the skeleton functor $\sk\co \cSet \to \Fil(\cSet)$ by Proposition~\ref{prop:Dadjunction}. However, the natural isomorphism
  \[
    \sk^n(K \boxtimes L) \cong \bigcup_{p+q=n} \sk^p(K) \boxtimes \sk^q(L)
  \]
  for cubical sets expresses precisely that $\sk$ is strong monoidal: it takes the product $\boxtimes$ on $\cSet$ to the corresponding Day convolution on $\Fil(\cSet)$. The right adjoint to a strong monoidal functor is automatically lax monoidal.
\end{proof}

\begin{rmk}
  \label{rmk:simplicialmonoidality}
  What works for cubes fails for simplices. While we can give a simplicial version of the successor functor $(-)^\der\co \Fil(\sSet) \to \sSet$, it is not lax monoidal. Additionally, the simplicial versions of Proposition~\ref{prop:Dlax} (and Theorem~\ref{thm:Dpullback}) seem to be considerably more difficult.

  The switch to cubes comes at a cost, however. Cubical sets are monoidal, but not symmetric monoidal, and so there is no question of upgrading this to a lax symmetric monoidal functor on filtered cubical sets. To obtain it, we will need to work homotopically.
\end{rmk}

\subsection{Products and weight}

\begin{prop}
  \label{prop:weightmonoidal}
  The Day convolution on $\Fil(\Spaces)$ is compatible with weight: it induces functors
  \[
    \Fil(\Spaces)^{\geq w_1} \times \Fil(\Spaces)^{\geq w_2} \to \Fil(\Spaces)^{\geq (w_1 + w_2)}
  \]
  for all $w_1$ and $w_2$.
\end{prop}

\begin{proof}
  Suppose $X$ is weight at least $w_1$ and $Y$ is weight at least $w_2$. We can again use Remark~\ref{rmk:relativeCW} to represent $X$ by a sequence $\dots \to X^0 \to X^1 \to \dots$ of CW-inclusions, with $X^m$ formed by attaching cells of dimension at least $w_1+m$, and similarly $Y^m$ is formed by attaching cells of dimension at least $w_2+m$. For such models, the $\infty$-categorical Day convolution $X \circledast Y$ is the filtered space which, in degree $n$, is modeled by the point-set Day convolution:
  \[
    (X \circledast Y)^n = \bigcup_{p+q=n} X^p \times Y^q \simeq \hocolim_{p+q \leq n} X^p \times Y^q.
  \]
  This filtered space $X \circledast Y$ is then again a sequence of relative CW-inclusions, with filtration $m$ only attaching cells of dimension at least $w_1+w_2+m$. Therefore, $X \circledast Y$ is weight at least $w_1+w_2$.
\end{proof}

\subsection{Coherent symmetric monoidality of successors}

By Corollary~\ref{cor:Dhomotopical}, $(-)^\der\co \Fil(\cSet) \to \cSet$ represents a functor of $\infty$-categories equivalent to the composite $\hocolim_{\mb Z} \tau_{\geq 0}(-)$. The lax monoidality of the cubical functor tells us that the induced successor functor is a lax monoidal functor of $\infty$-categories. However, because the product $\boxtimes$ on cubical sets is not symmetric monoidal, we cannot use it to understand the symmetric monoidal properties on filtered spaces. As mentioned in Remark~\ref{rmk:simplicialmonoidality}, switching from cubical sets to simplicial sets would make the problem worse.

\begin{thm}
  \label{thm:Dmonoidal}
  The successor functor $(-)^\der\co \Fil(\Spaces) \to \Spaces$ of $\infty$-categories is lax symmetric monoidal.
\end{thm}

\begin{proof}
The functor $\hocolim_{\mb Z}$ is lax symmetric monoidal for the Day convolution---for example, because it is equivalent to left Kan extension along the functor $\mb Z \to \ast$, by \cite[3.8]{nikolaus-yoneda}---and so it suffices to prove that $\tau_{\geq 0}$ is lax symmetric monoidal.
  
  However, Proposition~\ref{prop:weightmonoidal} implies that the functor $\Fil(\Spaces)_{\geq 0} \to \Fil(\Spaces)$ exhibits the source category as closed under the Day convolution, making the inclusion strong symmetric monoidal. Its right adjoint, $\tau_{\geq 0}$, is therefore lax monoidal, by \cite[7.3.2.7]{lurie-higheralgebra}.
\end{proof}

This theorem can be upgraded to a filtered version.

\begin{prop}
  \label{prop:filteredDmonoidal}
  The successor functor of $\infty$-categories $\Fil(\Spaces) \to \Fil(\Spaces)$ induced by $(-)^\der$ is lax symmetric monoidal.
\end{prop}

\begin{proof}
  Using Corollary~\ref{cor:Dhomotopical}, taking successors of the sequence $\dots \to X \to \sh X \to \dots$ is equivalent to applying $\hocolim_{\mb Z}$ to the tower of truncations
  \[
    \dots \to \tau_{\geq 0} X \to \tau_{\geq (-1)} X \to \dots
  \]
  and so the result follows from Proposition~\ref{prop:weightmonoidal} and the fact that $\hocolim_{\mb Z}\co \Fil(\Spaces) \to \Spaces$ is strong monoidal for the Day convolution.
\end{proof}

\section{Successor categories}
\label{sec:successorcat}

\subsection{Categories enriched in filtered spaces}

Because filtered spaces have both a Cartesian product and a Day convolution, the notion of an enriched category is potentially ambiguous.

\begin{defn}
  \label{def:enrichment}
  We say that $\mc C$ is \emph{enriched in filtered spaces} if it is enriched in $\Fil(\Spaces)$ under the Day convolution, and similarly for enrichment in nonnegatively filtered spaces.

  We write $\FMap_{\mc C}(X,Y)$ for the enriched mapping object, and refer to the elements in $\FMap_{\mc C}(X,Y)^n$ as maps that \emph{shift weight by $n$}.
\end{defn}

\begin{defn}
  \label{def:underlyingcat}
  For a category $\mc C$ enriched in filtered spaces, the \emph{underlying category associated to $\mc C$} has the same objects as $\mc C$, with mapping spaces consisting of the maps of weight zero:
  \[
    \Map_{\mc C}(X,Y) = \FMap_{\mc C}(X,Y)^0
  \]
  The \emph{limit category associated to $\mc C$} also has the same objects as $\mc C$, but with maps instead given by the (homotopy) colimit:
  \[
    \colim_{n \in \mb Z}\FMap_{\mc C}(X,Y)^n
  \]
\end{defn}

\subsection{Actions}

If $\mc C$ is enriched in (nonnegatively) filtered spaces, then for $Y \in \mc C$ and any appropriate $n$ we get a functor $\FMap_{\mc C}(-,Y)^n$, defined on the underlying category of $\mc C$. These functors may or may not be representable. If so, then they define endofunctors $Y \mapsto \sh^n Y$ of the underlying category of $\mc C$, in particular satisfying $\sh^n \circ \sh^m \cong \sh^{n+m}$ and with a natural transformation $\tau\co 1 \to \sh$ satisfying $\tau_{\sh X} = \sh(\tau_X)$.

In this section, we describe the converse situation.

\begin{defn}
  \label{defn:shiftaction}
  Suppose $\mc C$ is a category enriched in spaces. An \emph{action of the poset $\mb N$ on $\mc C$} consists of an endofunctor $\sh\co \mc C \to \mc C$ and a natural transformation $\tau\co 1 \to \sh$, together satisfying $\sh(\tau_X) = \tau_{\sh X}$.

  When a left adjoint to the functor $\sh$ exists, we denote it by $\shad$ and refer to this as an \emph{adjointable} action of the poset $\mb N$. If $\sh$ is a self-equivalence, we refer to this as an \emph{action of the poset $\mb Z$}.
\end{defn}

\begin{rmk}
  \label{rmk:shadtau}
  An adjointable action produces a natural transformation $\shad X \to X$, adjoint to $\tau\co X \to \sh X$.
\end{rmk}

\begin{rmk}
  \label{rmk:action}
  As described, an action is literally equivalent to an action of $\mb N$: a functor $\mb N \times \mc C \to \mc C$ satisfying associativity and unitality identities. A \emph{coherent} version of the following has been proved by Heine \cite{heine-enriched}.
\end{rmk}

\begin{prop}
  \label{prop:shiftenrichment}
  If $\mc C$ has an action of the poset $\mb N$, then the above definitions make $\mc C$ into the underlying category of a category enriched in nonnegatively filtered spaces. The maps of weight $n$ are given by $\Map_{\mc C}(X,\sh^n Y)$, and the composition operation $\Map_{\mc C}(Y,\sh^n Z) \times \Map_{\mc C}(X,\sh^m Y) \to \Map_{\mc C}(X,\sh^{n+m} Z)$ is given by
  \[
    (g,f) \mapsto \sh^m g \circ f.
  \]
  If $\mc C$ has an action of the poset $\mb Z$, then these definitions extend $\mc C$ to being enriched in filtered spaces.
\end{prop}

\begin{exam}
  \label{exam:filterenrich}
  If $\mc D$ is enriched in spaces, the category $\Fil(\mc D) = \Fun(\mb Z, \mc D)$ of filtered objects has shift endofunctors $\sh^k$ as in Definition~\ref{def:shift}, with a natural transformation $\tau\co X \to \sh X$. These satisfy $\sh(\tau) = \tau$ and automatically give $\Fil(\mc D)$ an enrichment in filtered spaces.

  Similarly, $\Fil^+(\mc D)$ is enriched in nonnegatively filtered spaces. If $\mc D$ has an initial object $\emptyset$, then this action is adjointable: the shift functor $\sh$ has a left adjoint given by
  \[
    \shad(X^0 \to X^1 \to \dots) = (\emptyset \to X^0 \to X^1 \to \dots).
  \]
\end{exam}

\subsection{Enriched limits and colimits}

One benefit of categories arising from an action is that (homotopy) limits and colimits in the underlying category are automatically upgraded to enriched versions.

\begin{prop}
  \label{prop:colimenrich}
  Suppose that $\mc C$ is a category with an action of $\mb N$. Then, for any (coherent) functor $A\co I \to \mc C$, any (homotopy) colimit $B$ in the underlying category of $\mc C$ also represents an enriched (homotopy) colimit in $\mc C$.
\end{prop}

\begin{proof}
  By definition, a colimit of the functor $A$ is an object $B$ with a natural isomorphism
  \[
    \Map_{\mc C}(B, X) \to \lim_I \Map_{\mc C}(A(i), X)
  \]
  Applied to the sequence of objects $X = \sh^n Y$, we find that there is a natural isomorphism of filtered spaces
  \[
    \FMap_{\mc C}(B,Y) \to \lim_I \FMap_{\mc C}(A(i), Y),
  \]
  and so $B$ is also an enriched colimit of the diagram $A$.

  Similarly, a homotopy colimit $B$ is an object with a natural weak equivalence
  \[
    \Map_{\mc C}(B, X) \simeq \holim_I \Map_{\mc C}(A(i), X)
  \]
  Applied to the sequence of objects $X = \sh^n Y$, we find that there is a natural equivalence of filtered spaces
  \[
    \FMap_{\mc C}(B,Y) \to \holim_I \FMap_{\mc C}(A(i), Y),
  \]
  and so $B$ is also an enriched homotopy colimit of the diagram $A$.
\end{proof}

\begin{prop}
  \label{prop:limenrich}
  Suppose that $\mc C$ is a category with an adjointable action of $\mb N$. Then, for any functor $F\co I \to \mc C$, any (homotopy) limit $B$ in the underlying category of $\mc C$ also represents an enriched (homotopy) limit in $\mc C$.
\end{prop}

\begin{proof}
  We first need to use the left adjoint of $\sh$ to first make the identification
  \[
    \FMap_{\mc C}(W, B)^n = \Map_{\mc C}(W, \sh^n B) \cong_{\mc C}((\shad)^n W, B).
  \]
  Once this is done, the same proof holds as for colimits.
\end{proof}

\begin{rmk}
  If we further have an action of the poset $\mb Z$, then we can exhibit this another way. The map $n \mapsto -n$ is an isomorphism $\mb Z \simeq \mb Z^\op$ of symmetric monoidal categories. This makes an action of $\mb Z$ on $\mc C$ the same as an action of $\mb Z$ on the opposite category $\mc C^\op$, which can be restricted to $\mb N$.
\end{rmk}

\begin{cor}
  \label{cor:actionproducts}
  If $\mc C$ has an action of the poset $\mb N$, then (homotopy) coproducts in $\mc C$ are automatically enriched (homotopy) coproducts.

  Similarly, if $\mc C$ has an adjointable action of the poset $\mb N$, then (homotopy) products in $\mc C$ are also enriched (homotopy) products.
\end{cor}

\subsection{Successor categories}

The successor functor is lax monoidal, so it converts enriched categories to enriched categories.

\begin{defn}
  \label{def:Dcategory}
  Suppose that $\mc C$ is enriched in (nonnegatively) filtered spaces. The category $\mc C^\der$ is the enriched category with the same objects as $\mc C$, but whose mapping spaces are obtained by taking successors:
  \[
    \Map_{\mc C^\der}(X,Y) = \FMap_{\mc C}(X,Y)^\der.
  \]
\end{defn}

\begin{rmk}
  There are multiple potential interpretations of this. If $\mc C$ is an $\infty$-category enriched in filtered spaces in the sense of Gepner--Haugseng \cite{gepner-haugseng-enriched}, we can also apply the homotopical successor functor to obtain an $\infty$-category enriched in (filtered) spaces. If instead $\mc C$ is an ordinary category enriched in filtered cubical sets, we can apply the strict version of taking successors to its mapping objects. This lifts the first construction if the mapping objects in $\mc C$ are fibrant, such as if $\mc C$ arose from a category enriched in filtered topological spaces.
\end{rmk}

\begin{prop}
  \label{prop:homotopy}
  The set $[X,Y]_{\mc C^\der}$ of maps $X \to Y$ in the homotopy category of $\mc C^\der$ is the image of $\pi_0 \FMap(X,Y)^0$ in $\pi_0 \FMap(X,Y)^1$.

  In particular, if the enrichment on $\mc C$ comes from an action of $\mb N$ on the underlying category, then there is a natural isomorphism
  \[
    [X,Y]_{\mc C^\der} = \im (\tau\co [X,Y]_{\mc C} \to [X,\sh Y]_{\mc C}).
  \]
\end{prop}

\begin{proof}
  We recall that
  \[
    [X,Y]_{\mc C^\der} = \pi_0 \Map_{\mc C^\der}(X,Y) = \pi_0 \FMap_{\mc C}(X,Y)^\der.
  \]
  The result is then a consequence of Theorem~\ref{thm:Dhomotopy}.
\end{proof}

\subsection{Products and coproducts}

The interaction of successors with products and coproducts allows us to conclude that the passage from $\mc C$ to $\mc C^\der$ often preserves them.

\begin{prop}
  \label{prop:productrep}
  Suppose that $\mc C$ is a category with an action of $\mb N$. Then, for any family of objects $A_i$, any (homotopy) coproduct $B$ in $\mc C$ also represents a (homotopy) coproduct in $\mc C^\der$.

  Similarly, suppose $\mc C$ is a category with an adjointable action of $\mb N$. Then, for any family of objects $X_i$, any (homotopy) product $Y$ in $\mc C$ also represents a (homotopy) product in $\mc C^\der$.
\end{prop}

\begin{proof}
  In the product case, by Corollary~\ref{cor:actionproducts}, we find that there is a natural isomorphism (resp. weak equivalence)
  \[
    \FMap_{\mc C}(B,Y) \to \prod \FMap_{\mc C}(A_i, Y)
  \]
  of filtered objects. Taking successors, Proposition~\ref{prop:Dproducts} shows that we get a natural isomorphism (resp. weak equivalence)
  \[
    \Map_{\mc C^\der}(B,Y) \to \prod \Map_{\mc C^\der}(A_i, Y).
  \]
  The dual proof holds for coproducts.
\end{proof}

\begin{cor}
  \label{cor:zeroobject}
  Suppose that $\mc C$ is a category with an adjointable action of $\mb N$, and that $Z$ represents a zero object of the $\infty$-category associated to $\mc C$. Then $Z$ also represents a zero object of the $\infty$-category associated to $\mc C^\der$.
\end{cor}

\begin{proof}
  A zero object $Z$ is an object that is both the homotopy product and homotopy coproduct of an empty set of objects. By Corollary~\ref{cor:actionproducts}, $Z$ also represents a zero object in the $\infty$-category associated to $\mc C^\der$, and so $\mc C^\der$ is pointed.
\end{proof}

\subsection{Homotopy pullbacks and pushouts}

\begin{prop}
  \label{prop:pullbackrep}
  Suppose that $\mc C$ is a category with an action of the poset $\mb N$, and that $X \to Z \leftarrow Y$ is a diagram of objects in $\mc C$. If there is a homotopy pullback $P$ of the resulting diagram $X \to \sh Z \leftarrow Y$, then $P$ becomes a homotopy pullback of the diagram $X \to Z \leftarrow Y$ in the successor category $\mc C^\der$.

  In particular, if $\mc C$ has homotopy pullbacks then so does $\mc C^\der$.
\end{prop}

\begin{proof}
  We recall from Theorem~\ref{thm:Dpullback} that there is a natural homotopy pullback diagram
  \[
    \begin{tikzcd}
      \left(\FMap_{\mc C}(A,X) \times_{\sh \FMap_{\mc C}(A,Z)} \FMap_{\mc C}(A,Y)\right)^\der \ar[r] \ar[d] &
      \FMap_{\mc C}(A,X)^\der \ar[d] \\
      \FMap_{\mc C}(A,Y)^\der \ar[r] &
      \FMap_{\mc C}(A,Z)^\der.
    \end{tikzcd}
  \]
  However, $\sh \FMap_{\mc C}(A,Z) = \FMap_{\mc C}(A, \sh Z)$, and there is a natural equivalence
  \[
    \FMap_{\mc C}(A,X) \times_{\FMap_{\mc C}(A, \sh Z)} \FMap_{\mc C}(A,Y) \simeq \FMap_{\mc C}(A,P)
  \]
  because $P$ is a homotopy pullback. Therefore, we can re-express our original diagram as a natural homotopy pullback diagram
  \[
    \begin{tikzcd}
      \Map_{\mc C^\der}(A,P) \ar[r] \ar[d] &
      \Map_{\mc C^\der}(A,X) \ar[d] \\
      \Map_{\mc C^\der}(A,Y) \ar[r] &
      \Map_{\mc C^\der}(A,Z).
    \end{tikzcd}
  \]
  Naturality in $A$ implies that $P$ is a homotopy pullback, as desired.
\end{proof}

A dual proof gives the following result.
\begin{prop}
\label{prop:pushoutrep}
  Suppose that $\mc C$ is a category with an adjointable action of the poset $\mb N$, and that $B \leftarrow A \to C$ is a diagram of objects in $\mc C$. If there is a homotopy pushout $Q$ of the resulting diagram $B \leftarrow \shad A \to C$, then $Q$ becomes a homotopy pushout of the diagram $B \leftarrow A \to C$ in the successor category $\mc C^\der$.

  In particular, if $\mc C$ has homotopy pushouts and an adjointable action of the poset $\mb N$, then $\mc C^\der$ has homotopy pushouts.
\end{prop}

\begin{rmk}
  It can be helpful to notationally distinguish between homotopy colimits in $\mc C$ and in $\mc C^\der$, especially because these categories share objects. For example, we will write $B \amalg^A C$ for the coproduct in $\mc C$ and $B \cocoprod^A C = B \coprod^{\shad A} C$ for the lift of the coproduct in $\mc C^\der$; $X/Y$ for a homotopy cofiber in $\mc C$ and $X\modmod Y = X / \shad Y$ for a homotopy cofiber in $\mc C^\der$; and similarly for fiber products.
\end{rmk}

\begin{exam}
  Suppose we have the category $\Fil(\mc D)$ of filtered objects and that $\mc D$ has homotopy pushouts. Then, for a diagram $B \leftarrow A \to C$ of objects of $\Fil(\mc D)$, we can represent the homotopy pushout $B \cocoprod^A C$ in $\mc C^\der$ by the filtered object
  \[
    \dots \to B^0 \amalg^{A^{-1}} C^0 \to  B^1 \amalg^{A^0} C^1 \to  B^2 \amalg^{A^1} C^2 \to \dots
  \]
  Similarly, for a diagram $X \to Z \leftarrow Y$, we can represent the homotopy pullback $X \ttimes_Z Y$ in $\mc C^\der$ by the filtered object
  \[
    \dots \to X^0 \times_{Z^1} Y^0 \to  X^1 \times_{Z^2} Y^1 \to  X^2 \times_{Z^3} Y^2 \to \dots
  \]
\end{exam}

\subsection{Completeness and cocompleteness}

\begin{prop}
  \label{prop:completecocomplete}
  Let $\kappa$ be an infinite cardinal.

  Suppose that $\mc C$ is a category with an action of the poset $\mb N$ and that $\mc C$ has $\kappa$-small limits. Then $\mc C^\der$ has $\kappa$-small homotopy limits.

  Dually, suppose that $\mc C$ is a category with an adjointable action of the poset $\mb N$ and that $\mc C$ has $\kappa$-small colimits. Then $\mc C^\der$ also has $\kappa$-small homotopy colimits.
\end{prop}

\begin{proof}
  By \cite[4.4.2.6]{lurie-htt}, having $\kappa$-small limits is equivalent to having pullbacks and $\kappa$-small products, and so the completeness statement follows from Proposition~\ref{prop:pullbackrep} and Proposition~\ref{prop:productrep}. Dually, having $\kappa$-small colimits is equivalent to having pushouts and $\kappa$-small coproducts, and so the cocompleteness statement follows from Proposition~\ref{prop:pushoutrep} and Proposition~\ref{prop:productrep}.
\end{proof}

\subsection{Nearly constant objects}

The definition of being nearly constant in Definition~\ref{def:nearlyconstantspace} extends to categories with an action of $\mb N$.

\begin{defn}
  \label{def:nearlyconstant}
  Suppose that $\mc C$ has an action of $\mb N$. We say that an object $Z$ is \emph{essentially constant} if the map $\tau\co Z \to \sh Z$ is an equivalence. We say that an object $X \in \mc C$ is \emph{nearly constant} if there is an object $Z$ and a diagram
  \[
    \begin{tikzcd}
      Z \ar[rr,"\tau"] \ar[dr] && \sh Z\\
      & X \ar[ur]
    \end{tikzcd}
  \]
  such that $\tau\co Z \to \sh Z$ is an equivalence. We refer to $Z$ as the \emph{value} of the nearly constant object.
\end{defn}

\begin{prop}
  \label{prop:nearlyconstanttoconstant}
  Suppose that $\mc C$ has an action of $\mb N$. If $X$ is nearly constant with value $Z$, then there is a natural equivalence
  \[
    \Map_{\mc C}(Y,Z) \simeq \Map_{\mc C^\der}(Y,X).
  \]
  Suppose that $\mc C$ has an action of $\mb Z$. If $X$ is nearly constant with value $Z$, then there is a natural equivalence
  \[
    \Map_{\mc C}(Z,Y) \simeq \Map_{\mc C^\der}(X,Y)
  \]
\end{prop}


\begin{proof}
  If $X$ is nearly constant with value $Z$, then applying $\FMap_{\mc C}(Y,-)$ to the definition we get a factorization:
  \[
    \begin{tikzcd}
      \FMap_{\mc C}(Y,Z) \ar[rr,"\tau"] \ar[dr] &&
      \FMap_{\mc C}(Y,\sh Z) \\
      &\FMap_{\mc C}(Y, X) \ar[ur]
    \end{tikzcd}
  \]
  This, by definition, makes $\FMap_{\mc C}(Y,X)$ nearly constant with value $\FMap_{\mc C}(Y,Z)$, and so we can apply Proposition~\ref{prop:Dnearlyconstant}. Moreover, this argument is natural in $Y$.

  If $\mc C$ has an action of $\mb Z$, then applying $\FMap_{\mc C}(-,Y)$ to the definition we get a factorization:
  \[
    \begin{tikzcd}
      \FMap_{\mc C}(\sh Z,Y) \ar[rr] \ar[dr] &&
      \FMap_{\mc C}(Z,Y) \\
      &\FMap_{\mc C}(X, Y) \ar[ur]
    \end{tikzcd}
  \]
  The action of $\mb Z$ lets us identify the top map with $\tau\co \FMap_{\mc C}(Z, \sh^{-1} Y) \to \sh \FMap_{\mc C}(Z, \sh^{-1} Y)$. Therefore, this makes $\FMap_{\mc C}(X, Y)$ essentially constant with value $\Map_{\mc C}(Z,\sh Y) \simeq \Map_{\mc C}(Z,Y)$, and so we can again apply Proposition~\ref{prop:Dnearlyconstant}. Similarly, this argument is natural in $Y$.
\end{proof}

\begin{prop}
  \label{prop:shiftnull}
  Suppose that $\mc C$ is pointed and has an action of $\mb Z$. Then, for any $X$, the levelwise homotopy cofiber $\sh X / X$ is nearly constant with value $0$.
\end{prop}

\begin{proof}
  The map $\tau\co \sh X / X \to \sh(\sh X / X) \simeq \sh^2 X / \sh X$ is nullhomotopic, and so the result follows from Corollary~\ref{cor:nearlycontractible}.
\end{proof}

\section{CW-complexes}
\label{sec:cw}

Our first example of the successor category is from unstable homotopy theory.

\subsection{Weighted connectivity}

\begin{defn}
  \label{def:weightconnectivity}
  A map $f\co Y \to Z$ of filtered spaces is \emph{weight $d$-connected} if, for all $m$, the map $Y^m \to Z^m$ is $(m+d)$-connected.
\end{defn}

\begin{rmk}
  In particular, a filtered space $Y$ is of weight at least $d$ if and only if the map $\tau\co Y \to \sh Y$ is weight $d$-connected.
\end{rmk}

\begin{defn}
  \label{def:weightcellular}
  A filtered space $X$ is \emph{weight cellular} if $X^{-1}$ is empty and, for all $m$, the map $X^{m-1} \to X^m$ is a relative cell inclusion formed by attaching cells of dimension less than or equal to $m$.
\end{defn}

\begin{rmk}
  For any CW-complex $X$, the filtered space $\sk(X)$ is weight cellular.
\end{rmk}

\begin{prop}
  \label{prop:weightlift}
  Suppose that $X$ is weight cellular and $f\co Y \to Z$ is weight $0$-connected. Then the induced map $\Map_{\Fil(\Spaces)}(X,Y) \to \Map_{\Fil(\Spaces)}(X,Z)$ is $0$-connected.
\end{prop}

\begin{proof}
  This is a straightforward obstruction-theoretic argument, by induction on degree and on cells.
\end{proof}

\begin{rmk}
  The definition of weight-connectivity of a map is, in the language of \cite{skeleta}, a \emph{connectivity structure} on filtered spaces---an unstable structure than can be the precursor to a $t$-structure or a weight structure. In these terms, Proposition~\ref{prop:weightlift} is that weight-cellular objects are $0$-skeletal, and the following result is \cite[3.5]{skeleta}. We will include a proof for completeness.
\end{rmk}

\begin{prop}
  \label{prop:weightheavylift}
  Suppose that $X$ is weight cellular and $f\co Y \to Z$ is weight $d$-connected. Then the induced map $\Map_{\Fil(\Spaces)}(X,Y) \to \Map_{\Fil(\Spaces)}(X,Z)$ is $d$-connected.
\end{prop}

\begin{proof}
  We recall the inductive characterization of connectedness: a map $U \to V$ of spaces is $0$-connected if it is surjective on $\pi_0$, and it is $d$-connected for $d > 0$ if it is $0$-connected and the map from $U$ to the homotopy pullback $\holim(U \to V \leftarrow U)$ is $(d-1)$-connected.

  We now prove the proposition by induction on $d$. The base case, when $d=0$, is precisely Proposition~\ref{prop:weightlift}. We need to show that if $Y \to Z$ is weight $d$-connected, the diagonal map
  \[
    \Map(X,Y) \to \holim\big(\Map(X,Y) \to \Map(X,Z) \leftarrow \Map(X,Y)\big)
  \]
  is $(d-1)$-connected. However, the functor $\Map(X,-)$ preserves homotopy limits; therefore, this is equivalent to asking that the map
  \[
    \Map(X,Y) \to \Map(X, \holim (Y \to Z \leftarrow Y)),
  \]
  induced by the diagonal $\Delta\co Y \to \holim(Y \to Z \leftarrow Y)$, is $(d-1)$-connected. However, the map $Y \to \holim(Y \to Z \leftarrow Y)$ is weight $(d-1)$-connected, and so the inductive hypothesis completes the proof.
\end{proof}

\subsection{CW-mapping spaces}
\begin{thm}
  \label{thm:CWmodel}
  For CW-complexes $X$ and $Y$, the map
  \[
    \Map_{\Fil(\Spaces)^\der}(\sk X, \sk Y) \to \Map_{\Spaces}(X,Y)
  \]
  is an equivalence.
\end{thm}

\begin{proof}
  The maps $\sh^m \sk Y \to \sh^{m+1} \sk Y$ are weight $m$-connected. By Proposition~\ref{prop:weightheavylift} the map
  \[
    \Map_{\Fil(\Spaces)}(\sk X, \sh^m \sk Y) \to 
    \Map_{\Fil(\Spaces)}(\sk X, \sh^{m+1} \sk Y)
  \]
  is $m$-connected, and so $\FMap_{\Fil(\Spaces)}(\sk X, \sk Y)$ is weight at least $0$ (Definition~\ref{def:nonnegative}). By Proposition~\ref{prop:Dnonneg}, the natural map
  \[
    \FMap_{\Fil(\Spaces)}(\sk X, \sk Y)^\der \to \hocolim_{\mb Z} \Map_{\Fil(\Spaces)}(\sk X, \sh^m \sk Y)
  \]
  is an equivalence.

  Let $cY$ be the constant filtered space with value $Y$. The maps $\sh^m Y \to cY$ are weight $m$-connected, and so by Proposition~\ref{prop:weightheavylift} the map
  \[
    \Map_{\Fil(\Spaces)}(\sk X, \sh^m \sk Y) \to 
    \Map_{\Fil(\Spaces)}(\sk X, c Y) \simeq \Map_{\Spaces}(X,Y)
  \]
  is $m$-connected. Taking homotopy colimits in $m$, we find that the map
  \[
    \hocolim_{\mb Z} \Map_{\Fil(\Spaces)}(\sk X, \sh^m \sk Y) \to 
    \Map_{\Spaces}(X,Y)
  \]
  is an equivalence, and hence so is the composite
  \[
    \FMap_{\Fil(\Spaces)}(\sk X, \sk Y)^\der \to \Map_{\Spaces}(X,Y)
  \]
  as desired.
\end{proof}

\section{Diagrams of successor categories}
\label{sec:diagrams}

\subsection{Functoriality and monoidality}

The following is relatively clear from functoriality of successors.

\begin{prop}
  \label{prop:Dfunctor}
  Suppose that $\mc C$ and $\mc D$ are categories enriched in filtered spaces and that $f\co \mc C \to \mc D$ is an enriched functor. Then there is an induced functor $f^\der\co \mc C^\der \to \mc D^\der$, and this preserves composition.
\end{prop}

Our goal is to improve this result to more complicated diagrams of enriched $\infty$-categories. However, by Theorem~\ref{thm:Dmonoidal}, the successor functor is lax symmetric monoidal, and by Proposition~\ref{prop:filteredDmonoidal} it extends to a lax symmetric monoidal endofunctor of filtered spaces. This allows us to appeal to \cite[10.6]{laxmonoidality} for the following result.

\begin{thm}
  \label{thm:mainfunctoriality}
  There is a lax symmetric monoidal functor
  \[
    (-)^\der\co \Cat_\infty^{\Fil(\Spaces)} \to \Cat_\infty^{\Spaces} \simeq \Cat_\infty
  \]
  sending an enriched $\infty$-category to its successor. Further, this extends to a functor
  \[
    (-)^\der\co \Cat_\infty^{\Fil(\Spaces)} \to \Cat_\infty^{\Fil(\Spaces)}.
  \]
\end{thm}

\begin{cor}
  \label{cor:Dstraightfunctorial}
  Suppose that $\mc J$ is an $\infty$-category. Any $\mc J$-indexed diagram $\{\mc C_j\}$ of $\infty$-categories enriched in filtered spaces gives rise to a $\mc J$-indexed diagram $\{(\mc C_j)^\der\}$.
\end{cor}

Lax symmetric monoidality further implies that taking successors respects monoidal structures on enriched categories.

\begin{prop}
  \label{prop:Dcoherentmonoidal}
  Suppose that $\mc O$ is an $\infty$-operad and that $\mc C$ has the structure of an $\mc O$-monoidal $\infty$-category enriched in $\Fil(\Spaces)$. Then $\mc C^\der$ has the structure of an $\mc O$-monoidal $\infty$-category, and this is natural in $\mc C$.
\end{prop}

\begin{proof}
  Because the successor functor $(-)^\der\co \Cat_\infty^{\Fil(\Spaces)} \to \Cat_\infty$ is lax symmetric monoidal, we get an induced functor
  \[
    \Alg_{\mc O}(\Cat_\infty^{\Fil(\Spaces)}) \to \Alg_{\mc O}(\Cat_\infty).
  \]
  However, $\mc O$-algebras in (enriched) $\infty$-categories are equivalent to $\mc O$-monoidal $\infty$-categories \cite[2.4.2.6]{lurie-higheralgebra}.
\end{proof}

\subsection{(Co)limit-preservation}

\begin{prop}
  \label{prop:functoriallimits}
  Suppose that $\mc C$ and $\mc D$ have actions of $\mb Z$, and that $f\co \mc C \to \mc D$ is a shift-preserving functor: there are natural, compatible equivalences
  \[
    f(\sh^n X) \to \sh^n f(X).
  \]
  Then the construction $f \mapsto f^\der$ respects limits and colimits as follows.
  \begin{enumerate}
  \item Any coproduct in $\mc C$ that is preserved by $f$ is also preserved by $f^\der$.
  \item Any product in $\mc C$ that is preserved by $f$ is also preserved by $f^\der$.
  \item If $\mc C$ has pushouts and $f$ preserves them, then so does $f^\der$.
  \item If $\mc C$ has pullbacks and $f$ preserves them, then so does $f^\der$.
  \item If $\mc C$ has $\kappa$-small colimits and $f$ preserves them, then so does $f^\der$.
  \item If $\mc C$ has $\kappa$-small limits and $f$ preserves them, then so does $f^\der$.
  \end{enumerate}
\end{prop}

\begin{proof}
  The functor $f^\der$ preserves any (co)products by Proposition~\ref{prop:productrep}. If $f$ preserves pushouts or pullbacks, then so does $f^\der$ by Propositions~\ref{prop:pullbackrep} and \ref{prop:pushoutrep}. Finally, for any infinite cardinal $\kappa$, all $\kappa$-small colimits are generated by pushouts and $\kappa$-small coproducts; similarly limits.
\end{proof}



\subsection{(Co)cartesian fibrations}

To show that successors respect more complicated adjunction structures, it will be convenient to show compatibility with (co)Cartesian fibrations.

\begin{defn}
  \label{def:enrichedcocart}
Suppose that $p\co \mc C \to \mc D$ is a functor between categories enriched in filtered spaces. We say that a morphism $g \co X \to Y$ in $\mc C$ is \emph{enriched $p$-cocartesian} if, for all $Z \in \mc C$, the diagram
\[
  \begin{tikzcd}
    \FMap_{\mc C}(Y, Z) \ar[r] \ar[d] &
    \FMap_{\mc C}(X, Z) \ar[d] \\
    \FMap_{\mc D}(pY, pZ) \ar[r] &
    \FMap_{\mc D}(pX, pZ) 
  \end{tikzcd}
\]
is a homotopy pullback diagram of filtered spaces. We say that the functor $p$ is \emph{essentially enriched cocartesian} if, for every $X$ and every map $\bar g\co pX \to Z$, there exists an enriched $p$-cocartesian map $g\co X \to Y$ whose image under $p$ is equivalent (under $pX$) to $\bar g$.

Similarly, we have a dual notion of enriched $p$-cartesian maps and essentially enriched cartesian functors.
\end{defn}

\begin{prop}
  \label{prop:Dcocartesianmap}
  Suppose $p\co \mc C \to \mc D$ is a functor between categories enriched in filtered spaces. If $g\co X \to Y$ in $\mc C$ is enriched $p$-cocartesian and, for all $Z$, the map $\FMap_{\mc C}(X,Z) \to \FMap_{\mc D}(pX,pZ)$ satisfies the defibration condition (Definition~\ref{def:defibrationcond}), then the image of $g$ in $\mc C^\der$ is (essentially) $p^\der$-cocartesian.
\end{prop}

\begin{proof}
  For all objects $Z$, we have a homotopy pullback diagram
  \[
    \begin{tikzcd}
      \FMap_{\mc C}(Y,Z) \ar[r] \ar[d] &
      \FMap_{\mc C}(X,Z) \ar[d] \\
      \FMap_{\mc D}(pY,pZ) \ar[r] &
      \FMap_{\mc D}(pX,pZ) 
    \end{tikzcd}
  \]
  of filtered spaces, and by assumption the right-hand map always satisfies the defibration condition. By Proposition~\ref{prop:Ddefibrationpullback}, for any $Z$ we have a homotopy pullback diagram
  \[
    \begin{tikzcd}
      \Map_{\mc C^\der}(Y,Z) \ar[r] \ar[d] &
      \Map_{\mc C^\der}(X,Z) \ar[d] \\
      \Map_{\mc D^\der}(p^\der(Y),p^\der(Z)) \ar[r] &
      \Map_{\mc D^\der}(p^\der(X),p^\der(Z)).
    \end{tikzcd}
  \]
  By \cite[2.4.4.3]{lurie-htt}, the map $g\co X \to Y$ is essentially $p^\der$-cocartesian.
\end{proof}

By considering opposite categories, we will arrive at a corresponding result for cartesian maps.

\begin{prop}
  \label{prop:Dcocartesian}
  Suppose $p\co \mc C \to \mc D$ is a functor between categories enriched in filtered spaces. If $p$ is enriched cocartesian and, for all $X$ and $Y$, the map $\FMap_{\mc C}(X,Y) \to \FMap_{\mc D}(pX,pY)$ satisfies the defibration condition (Definition~\ref{def:defibrationcond}), then $p^\der\co \mc C^\der \to \mc D^\der$ is an essentially cocartesian functor of $\infty$-categories.
\end{prop}

The defibration condition is trivial for constant filtered objects, giving rise to the following result. In unstraightened terms, it again says that a $\mc J$-indexed diagram of enriched categories becomes a $\mc J$-indexed diagram after taking successors.

\begin{cor}
  \label{cor:Dunstraightfunctorial}
  Suppose that $\mc J$ is a category enriched in spaces, regarded as enriched in constant filtered spaces. Given an enriched (co)cartesian functor $p\co \mc C \to \mc J$, the resulting functor $p^\der\co \mc C^\der \to \mc J$ is (co)cartesian.
\end{cor}

For the purposes of applying these results, it will be convenient to be able to identify when maps are automatically enriched (co)cartesian.
\begin{prop}
  \label{prop:enrichedcartesianshift}
  Suppose that $p\co \mc C \to \mc D$ is a shift-preserving functor between categories with an action of $\mb Z$, and $g\co X \to Y$ is a map in $\mc C$.

  If $g$ is $p$-cocartesian, then $\sh(g)$ is $p$-cocartesian and $g$ is enriched $p$-cocartesian. 
  Similarly, if $g$ is $p$-cartesian, then $\sh(g)$ is $p$-cartesian and $g$ is enriched $p$-cartesian. 
\end{prop}

\begin{proof}
  The notion of an action of $\mb Z$ is dualizable, and by considering the opposite category we find that it is sufficient to prove the cocartesian case.
  
  If $g$ is $p$-cocartesian, then for any $Z$ we get a homotopy pullback diagram
  \[
    \begin{tikzcd}
      \Map_{\mc C}(Y, Z) \ar[r] \ar[d] &
      \Map_{\mc C}(X, Z) \ar[d] \\
      \Map_{\mc D}(pY, pZ) \ar[r] &
      \Map_{\mc D}(pX, pZ) 
    \end{tikzcd}
  \]
  Substituting $Z = \sh^{-1} W$ for an arbitrary $W$, and using the fact that $p$ is shift-preserving, we find that we have a homotopy pullback diagram expressing that $\sh g$ is $p$-cocartesian. Applying this to arbitrary shifts $Z = \sh^m W$, we find that $g$ is also enriched cocartesian.
\end{proof}

\begin{rmk}
  There are some mild weakenings of the hypotheses above which still make the result above hold. However, these are more complicated to state, and in the absence of need they are left to the reader.
\end{rmk}

\subsection{Adjunctions}

The following are straightforward from the definitions.

\begin{prop}
  \label{prop:enrichedadjunction}
  Suppose that $\mc C$ and $\mc D$ are categories enriched in filtered spaces, and that we have an enriched adjunction between them: functors $f\co \mc C \to \mc D$ and $g\co \mc D \to \mc C$ with a natural isomorphism
  \[
    \FMap_{\mc C}(X,gY) \to \FMap_{\mc D}(fX,Y).
  \]
  Then there is an induced adjunction between the associated successor functors $f^\der\co \mc C^\der \to \mc D^\der$ and $g^\der\co \mc D^\der \to \mc C^\der$.
\end{prop}

\begin{cor}
  \label{cor:shiftadjunction}
  Suppose that $\mc C$ and $\mc D$ have actions of $\mb Z$ and that we have a shift-preserving functor $g\co \mc D \to \mc C$ with a left adjoint $f$. Then there is an induced adjunction between the associated successor functors $f^\der\co \mc C^\der \to \mc D^\der$ and $g^\der\co \mc D^\der \to \mc C^\der$.
\end{cor}

The following is a coherent version of this statement.

\begin{prop}
  \label{prop:coherentadjunction}
  An adjunction of $\infty$-categories enriched in filtered spaces gives rise to an adjunction of successor $\infty$-categories.
\end{prop}

\begin{proof}
  An adjunction determines and is determined by a functor $p\co \mc E \to \Delta^1$ that is both cocartesian and cartesian \cite[5.2.2.1]{lurie-htt}, with $\mc C = p^{-1}(0)$ and $\mc D = p^{-1}(1)$. An enriched adjunction correspondingly determines and is determined by a lift of this to an enriched functor that is enriched cartesian and cocartesian, and thus we can apply Corollary~\ref{cor:Dunstraightfunctorial}.
\end{proof}

\begin{rmk}
  If the adjunction between $\mc C$ and $\mc D$ is monadic, it is often tempting to try and immediately conclude that the adjunction between $\mc C^\der$ and $\mc D^\der$ is also monadic. However, even if the right adjoint $g\co \mc D \to \mc C$ is conservative, the induced functor $g^\der\co \mc D^\der \to \mc C^\der$ is often not necessarily so. We do, however, have the following.
\end{rmk}

\begin{prop}
  \label{prop:weakmonadicity}
  Suppose that $\mc C$ and $\mc D$ have actions of $\mb Z$ and that we have a shift-preserving functor $g\co \mc D \to \mc C$ that admits a left adjoint $f$. If $\mc D$ admits, and $g$ preserves, countable limits, then the induced functor
  \[
    \mc D^\der \to \Alg_{(gf)^\der}(\mc C^\der)
  \]
  has a fully faithful left adjoint.
\end{prop}

\begin{proof}
  The functor $\mc D^\der \to \mc C^\der$ preserves geometric realizations by Proposition~\ref{prop:functoriallimits}, and so this follows from \cite[4.7.3.13]{lurie-higheralgebra}.
\end{proof}

\begin{rmk}
  The essential image of $\Alg_{(gf)^\der}(\mc C^\der)$ is contained in the closure of the image of $\mc C^\der$ under geometric realizations. If $\mc C^\der$ has coproducts and $\mc D$ is cocomplete, this implies that the essential image of $\Alg_{(gf)^\der}(\mc C^\der)$ is the colimit-closure of the essential image of $\mc C^\der$.
\end{rmk}

\section{Stable categories}
\label{sec:stable}

\subsection{Zero objects, cofibers, and suspensions}

By Corollary~\ref{cor:zeroobject}, if $\mc C$ has a zero object, so does $\mc C^\der$. The homotopy pushout formula of Proposition~\ref{prop:pushoutrep} allows us to explicitly identify homotopy cofibers and suspensions in $\mc C^\der$.

\begin{prop}
  \label{prop:Dcofiber}
  Suppose that $\mc C$ is pointed and has an adjointable action of $\mb N$. Given a map $f\co X \to Y$ in $\mc C$, the object $X / \shad Y$ in $\mc C$ represents the homotopy cofiber $X \modmod Y$ of $f$ in $\mc C^\der$.
\end{prop}

\begin{cor}
  \label{cor:Dsuspension}
  Suppose that $\mc C$ is pointed and has an adjointable action of $\mb N$. For any $X \in \mc C$, the object $\Sigma \shad X$ in $\mc C$ represents the suspension $\sumsum X$ of $X$ in $\mc C^\der$.
\end{cor}

\subsection{Stability}

Recall from \cite[1.4.2.27]{lurie-higheralgebra} that an $\infty$-category $\mi C$ is \emph{stable} if and only if it is pointed, has finite homotopy colimits, and that the suspension functor $\Sigma\co \mi C \to \mi C$ is a self-equivalence. (We will refer to a category enriched in spaces as \emph{stable} if the associated $\infty$-category is also stable.)

\begin{thm}
  \label{thm:stabletostable}
  Suppose that $\mc C$ has an action of $\mb Z$, and that the underlying category of $\mc C$ is stable. Then $\mc C^\der$ is stable.
\end{thm}

\begin{proof}
  We need to prove that the $\infty$-category associated to $\mc C^\der$ has finite homotopy colimits, a zero object, and that the suspension functor $\sumsum$ of $\mc C^\der$ is an autoequivalence.

  The category $\mc C^\der$ has a zero object by Corollary~\ref{cor:zeroobject}, and finite homotopy colimits by Proposition~\ref{prop:completecocomplete}.

  By Corollary~\ref{cor:Dsuspension}, the suspension functor $\sumsum$ for $\mc C^\der$ lifts to the composite 
  \[
    X \mapsto \Sigma (\sh^{-1} X)
  \]
  of enriched functors on $\mc C$. Because $\sh$ and $\Sigma$ are both autoequivalences of the associated enriched $\infty$-category of $\mc C$ by assumption, so is $\sumsum$.
\end{proof}

\subsection{Preserved cofibers}

Certain special cofiber sequences in $\mc C$ may be preserved by the passage to $\mc C^\der$.

\begin{prop}
  \label{prop:cofiberpres}
  Let $\mc C$ be stable with an action of the poset $\mb Z$. Suppose $P \to Q \to R \to \Sigma P$ is a cofiber sequence in $\mc C$, and the map $R \to \Sigma P$ lifts to a map $R \to \sh^{-1} \Sigma P$. Then $P \to Q \to R \to \Sigma P$ is still a cofiber sequence in $\mc C^\der$.
\end{prop}

\begin{proof}
  Let $\tilde R$ denote the homotopy pullback of the diagram $R \to \Sigma P \leftarrow \sh^{-1} \Sigma P$ in $\mc C$, so that we have a diagram of fiber sequences
  \[
    \begin{tikzcd}
      \sh^{-1} P \ar[r] \ar[d] &
      Q \ar[r] \ar[d,equals] &
      \tilde R \ar[r] \ar[d] &
      \sh^{-1} \Sigma P \ar[r] \ar[d] &
      \Sigma Q \ar[d,equals]\\
      P \ar[r] &
      Q \ar[r] &
      R \ar[r] \ar[ur,dotted]&
      \Sigma P \ar[r] &
      \Sigma Q
    \end{tikzcd}
  \]
  in $\mc C$. The lift $R \to \sh^{-1} \Sigma P$ lifts to a splitting of the map $\tilde R \to R$, so that $\tilde R \simeq R \oplus Y$ where $Y$ is the homotopy fiber. By definition of $\tilde R$, $Y \simeq fib(\sh^{-1} \Sigma P \to \Sigma P)$, and hence $Y \simeq P / \sh^{-1} P$. By Proposition~\ref{prop:shiftnull}, $Y^\der \simeq 0$, and so $\tilde R^\der \to R^\der$ is an equivalence by Proposition~\ref{prop:productrep}.

  The map $\tilde R \to R$ therefore becomes an equivalence in $\mc C^\der$. However, $\tilde R$ is the cofiber of $\sh^{-1} P \to Q$ in $\mc C$, and therefore in $\mc C^\der$ it represents the cofiber of the map $P \to Q$ by Proposition~\ref{prop:pushoutrep}.
\end{proof}

\subsection{Compact objects}

\begin{prop}
  \label{prop:compactpres}
  Let $\mc C$ be stable with an action of $\mb Z$, with filtered homotopy colimits that are preserved by $\sh$. If $X$ is a compact object in $\mc C$, then the image of $X$ is also compact in $\mc C^\der$.
\end{prop}

\begin{proof}
  As $\mc C^\der$ is stable by Theorem~\ref{thm:stabletostable}, we can write $\Map_{\mc C^\der}(X,-) = \Omega^\infty \circ \map_{\mc C^\der}(X,-)$. Because $\Omega^\infty$ preserves filtered homotopy colimits, it will suffice to prove that the functor
  \[
    \map_{\mc C^\der}(X,-)\co \mc C^\der \to \Sp
  \]
  preserves all hocolims. By \cite[4.4.2.7]{lurie-htt}, this functor preserves all hocolims if and only if it preserves homotopy pushouts and coproducts.

  Since $\mc C$ is stable, the functor $\map_{\mc C^\der}(X,-)$ automatically preserves homotopy pushouts and finite coproducts.

  Compactness of $X$ implies that, for any filtered diagram $Y_i$ of objects, the natural map
  \[
    \hocolim \Map_{\mc C}(X, Y_i) \to \Map_{\mc C}(X, \hocolim Y_i)
  \]
  is an equivalence. The same is true for the diagrams $\sh^m Y_i$, so we get an equivalence
  \[
    \hocolim \FMap_{\mc C}(X, Y_i) \to \FMap_{\mc C}(X, \hocolim Y_i)
  \]
  of filtered spaces. By Proposition~\ref{prop:Dfilteredhocolim}, upon taking successors we get an equivalence
  \[
    \hocolim \Map_{\mc C^\der}(X, Y_i) \to \Map_{\mc C^\der}(X, \hocolim Y_i),
  \]
  and hence the same is true for the mapping spectrum $\map_{\mc C^\der}(X,-)$. Any coproduct $\oplus Z_j$ can be written as a filtered colimit of its finite summands, and this then specializes to an equivalence
  \[
    \oplus \map_{\mc C^\der}(X, Z_j) \to \map_{\mc C^\der}(X, \oplus Z_j).
  \]
  Because coproducts in $\mc C$ and in $\mc C^\der$ coincide by Proposition~\ref{prop:productrep}, we are done.
\end{proof}


\subsection{Associated graded}

We pause here to record the interaction between homotopy colimits in $\Fil(\mc D)^\der$ and the associated graded.

\begin{prop}
  \label{prop:gradedcofiber}
  Suppose that $\mc D$ is stable. Given a map $f\co P \to Q$ in $\Fil(\mc D)$, with representative $Q / \sh^{-1} P$ for the cofiber $Q \modmod P$ in $\Fil(\mc D)^\der$, there is an equivalence of associated gradeds:
  \[
    \gr(Q / \sh^{-1} P) \simeq \gr Q \oplus \Sigma \sh^{-1} \gr(P).
  \]
  Similarly, given a map $f\co P \to Q$ in $\Fil^+(\mc D)$, with representative $Q / \shad P$ for the cofiber $Q \modmod P$ in $\Fil^+(\mc D)^\der$, there is an equivalence of associated gradeds:
  \[
    \gr(Q / \shad P) \simeq \gr Q \oplus \Sigma \shad \gr(P).
  \]
\end{prop}

\begin{proof}
  The associated graded functor preserves cofibers, and so there is a cofiber sequence
  \[
    \gr(\shad P) \to \gr(Q) \to \gr(Q / \sh^{-1} P)
  \]
  of $\mb Z$-graded objects in $\mc C$. However, in the sequence
  \[
    \gr(\shad P) \to \gr(P) \to \gr(Q),
  \]
  the first map is nullhomotopic, and so the cofiber of the composite map decomposes as desired. The same proof holds for nonnegatively filtered objects.
\end{proof}

\begin{cor}
  \label{cor:generatedgraded}
  Suppose that $\mc P$ is a class of objects of $\Fil(\mc D)$, and $X$ any object of $\Fil(\mc D)^\der$ which can be generated by $\mc P$ under homotopy colimits. Then $X$ has a representative whose associated graded is a sum of terms of the form $\Sigma^i \sh^{-j} \gr(P)$ for $P \in \mc P$.

  Similarly, suppose that $\mc P$ is a class of objects of $\Fil^+(\mc D)$, and $X$ any object of $\Fil^+(\mc D)^\der$ which can be generated by $\mc P$ under homotopy colimits. Then $X$ has a representative whose associated graded is a sum of terms of the form $\Sigma^i (\shad)^j \gr(P)$ for $P \in \mc P$.
\end{cor}

\begin{proof}
  Let $G(\mc P)$ be the class of objects of $\mc C^\der$ which admit such a representative. Clearly $\mc P \subset G(\mc P)$. We find that $G(\mc P)$ is closed under coproducts because $\gr(\oplus X_i) \simeq \oplus \gr(X_i)$, and $G(\mc P)$ is closed under cofibers by Proposition~\ref{prop:gradedcofiber}. Therefore, $G(\mc P)$ contains the full subcategory of $\mc C^\der$ generated by $\mc P$, as desired. The same proof holds for nonnegatively filtered objects.
\end{proof}

\section{Filtered spectra}
\label{sec:filteredspectra}

\subsection{Spectral sequences}
\label{sec:spectral-sequences}

Associated to a filtered spectrum $X$, there is a bigraded collection of homotopy groups $\pi_* X^*$. There are also long exact sequences
\[
  \begin{tikzcd}
    \pi_t X^{n-1} \ar[rr] && \pi_t X^n \ar[dl]\\
    & \pi_t (X^n / X^{n-1}) \ar[ul,"\circ" marking]
  \end{tikzcd}
\]
with the bottom-left arrow with circle indicating a degree-changing map from $\pi_t$ to $\pi_{t-1}$. Assembled over $n$, these form an \emph{unrolled exact couple} in the sense of Boardman \cite{boardman-conditional}. Accordingly, we have a spectral sequence whose $E_1$-term consists of the bigraded groups $\pi_t (X^n / X^{n-1})$, roughly attempting to compute the difference between $\pi_* \lim(X^n)$ and $\pi_* \colim X^n$.

In this section our goal is to show that the successor category of filtered spectra also has bigraded homotopy groups, and these compute the derived exact couple.

\begin{rmk}
  \label{rmk:indexing}
  The indexing in the following sections is chosen to match the synthetic indexing from \cite{pstragowski-synthetic}, even though it has less immediacy in the filtered setting.
\end{rmk}

\subsection{The successor category of filtered spectra}

The category $\Fil(\Sp)$ of filtered spectra has an associated successor category $\Fil(\Sp)^\der$.

\begin{defn}
  The functor $\nu\co \Sp \to \Fil(\Sp)^\der$ is the composite of the functor $F_0$ (Definition~\ref{def:freeeval}) with the functor $\Fil(\Sp) \to \Fil(\Sp)^\der$.
\end{defn}

In particular,
\[
  (\nu X)^n =
  \begin{cases}
    X &\text{if }n \geq 0,\\
    0 &\text{otherwise},
  \end{cases}
\]
with nontrivial structure maps being the identity maps of $X$.

\begin{defn}
  \label{def:bishift}
  For an object $X \in \Fil(\Sp)^\der$, we define the \emph{bigraded suspension} by
  \[
    \Sigma^{t,w} X \simeq \sumsum^{t} \sh^w X.
  \]
\end{defn}

\begin{rmk}
  \label{rmk:bishift}
  By Corollary~\ref{cor:Dsuspension}, $\sumsum^{t}$ is represented by the functor $\Sigma^{t} \circ \sh^{-t}$, and so $\Sigma^{t,w} X$ lifts to the filtered spectrum
  \[
    \Sigma^t \sh^{w-t} X.
  \]
  At more length:
  \begin{align*}
    \sumsum X &= \Sigma^{1,0} X\\
    \sh X &= \Sigma^{0,1} X\\
    \Sigma X &= \Sigma^{1,1} X
  \end{align*}
\end{rmk}

\subsection{Spheres and natural homotopy groups}

\begin{defn}
  \label{def:bisphere}
  The \emph{sphere $S^{0,0}$} is the filtered spectrum $\nu \mb S$, and the \emph{bigraded sphere} $S^{t,w}$ is the filtered spectrum $\Sigma^{t,w} S^{0,0}$.
\end{defn}

\begin{rmk}
  In particular, $\sumsum S^{t,w} \simeq S^{t+1, w}$, while $\sh S^{t,w} \simeq S^{t,w+1}$.
\end{rmk}

\begin{prop}
  \label{prop:biformulations}
  There are equivalences
  \[
    S^{t,w} \simeq F_{t-w} S^t \simeq \sh^{w-t} \nu S^t \simeq \sumsum^{t-w} \nu S^w.
  \]
\end{prop}

\begin{proof}
  This follows from Remark~\ref{rmk:bishift}.
\end{proof}

\begin{defn}
  \label{def:binatural}
  For an object $X \in \Fil(\Sp)^\der$, the \emph{bigraded homotopy groups} (or \emph{natural homotopy groups}) by
  \[
    \pi_{t,w} X \simeq [S^{t,w}, X].
  \]
\end{defn}

\subsection{Quotients by $\tau$}

\begin{defn}
  \label{def:modtau}
  For a filtered spectrum $X$, the \emph{quotient by $\tau$}, denoted $X\modmod\tau$, is the cofiber of the map
  \[
    \Sigma^{0,-1} X \to X
  \]
  in $\Fil(\Sp)^\der$ induced by the natural map $\tau\co X \to \sh X$.

  In particular, we write $C\tau$ for $S^{0,0}\modmod\tau$.
\end{defn}

\begin{rmk}
  \label{rmk:Ctaurepresentative}
  In particular, $X\modmod \tau$ lifts to the filtered spectrum $X / sh^{-2} X$. The filtered spectrum $C\tau$ is equivalent to $S^0$ in filtrations $0$ and $1$, and $0$ in all other degrees.
\end{rmk}

\begin{defn}
  \label{def:e2homotopy}
  We refer to the bigraded homotopy groups $\pi_{t,w}(X\modmod\tau)$ as the \emph{$E_2$-homotopy groups} of $X$.
\end{defn}

Applying $\pi_{\ast,\ast}$ to the cofiber sequence $\Sigma^{0,-1} X \to X \to X \modmod \tau$, we obtain the following.
\begin{prop}
  \label{prop:spiralexact}
  There exists a natural long exact sequence
  \[
    \dots \to \pi_{t,w+1} X \to \pi_{t,w} X \to \pi_{t,w} (X\modmod\tau) \to \pi_{t-1,w+1} X \to \dots
  \]
  called the \emph{spiral exact sequence}.
\end{prop}

\begin{prop}
  \label{prop:Ctaurepresents}
  The $E_2$-homotopy groups are representable in $\Fil(\Sp)^\der$: there are natural isomorphisms
  \[
    \pi_{t,w}(X\modmod\tau) \cong [\Sigma^{t-1,w+1} C\tau, X]_{\Fil(\Sp)^\der}.
  \]
\end{prop}

\begin{proof}
  The defining cofiber sequence $S^{0,-1} \to S^{0,0} \to C\tau$ determines cofiber sequences
  \[
    S^{t-1,w} \xrightarrow{\tau} S^{t-1,w+1} \to \Sigma^{t-1,w+1} C\tau.
  \]
  Applying $\Map_{\Fil(\Sp)^\der}(-,X)$ to this gives a natural fiber sequence
  \[
    \Map_{\Fil(\Sp)^\der}(\Sigma^{t-1,w+1} C\tau, X) \to \Map_{\Fil(\Sp)^\der}(S^{t-1,w+1}, X) \xrightarrow{\tau} \Map_{\Fil(\Sp)^\der}(S^{t-1,w}, X)
  \]
  However, naturality of $\tau$ implies that the second of these maps is equivalent to the map $\Map(S^{t-1,w}, \Sigma^{0,-1} X) \to \Map(S^{t-1,w}, X)$ induced by $\tau\co \Sigma^{0,-1} X \to X$. Therefore, the fiber is equivalent to $\Map(S^{t,w}, X\modmod\tau)$, and taking $\pi_0$ implies the desired result.
\end{proof}

\begin{rmk}
  The spiral exact sequence is alternatively induced by a cofiber sequence $\Sigma^{t-1,w+1} C\tau \to S^{t,w} \to S^{t,w+1}.$
\end{rmk}

\subsection{The derived couple}

In this section we will identify the natural and $E_2$-homotopy groups of $X$ in terms of the homotopy groups of the underlying filtered spectrum. In particular, the spiral exact sequence for $X$ can be identified with the derived exact couple associated to its homotopy spectral sequence from Section~\ref{sec:spectral-sequences}.

\begin{prop}
  \label{prop:naturalidentify}
  For a filtered spectrum $X$, we have a natural identification
  \[
    \pi_{t,w} X \cong \im\left(\pi_t X^{t-w} \to \pi_t X^{t-w+1}\right).
  \]
\end{prop}

\begin{proof}
  By Proposition~\ref{prop:homotopy} we have
  \[
    \pi_{t,w} X \cong \im \left([S^{t,w}, X]_{\Fil(\Sp)} \to [S^{t,w}, \sh X]_{\Fil(\Sp)}\right).
  \]
  Since $S^{t,w}$ is identified with $F_{t-w} S^t$, we have $\Map_{\Fil(\Sp)}(S^{t,w}, Y) \simeq \Map_{\Sp}(S^t, Y^{t-w})$, and thus $[S^{t,w}, Y]_{\Fil(\Sp)} \cong \pi_t Y^{t-w}$. Applying this identification to the map $X \to \sh X$ gives the desired result.
\end{proof}

\begin{prop}
  \label{prop:indexinghell}
  For a filtered spectrum $X$, there is an isomorphism between the $E_2$-homotopy group $\pi_{t,w} (X\modmod\tau)$ and the homology of the chain complex
  \[
    \pi_{t+1} (X^{t-w+1}/X^{t-w})
    \xrightarrow{\partial} \pi_t (X^{t-w}/X^{t-w-1})
    \xrightarrow{\partial} \pi_{t-1} (X^{t-w-1}/X^{t-w-2})
  \]
  at the middle displayed term.
\end{prop}

\begin{proof}
  We recall that $X\modmod\tau$ is modeled in $\Fil(\Sp)$ by the cofiber $X / \sh^2 X$, which is $X^m / X^{m-2}$ in degree $m$.

  We now consider the following diagram.
  \[
    \begin{tikzcd}[column sep=tiny]
       \pi_t(X^{t-w-1}/X^{t-w-2}) \ar[dr]\\
      & \pi_t(X^{t-w}/X^{t-w-2}) \ar[dd] \ar[dl]\\
       \pi_t(X^{t-w}/X^{t-w-1}) \ar[uu, "\circ" marking, "\partial"] \ar[dr]\\
      & \pi_t(X^{t-w+1}/X^{t-w-1}) \ar[dl] \\
       \pi_t(X^{t-w+1}/X^{t-w}) \ar[uu, "\circ" marking, "\partial"]
    \end{tikzcd}
  \]
  Here the bottom and top triangles are long exact sequences (with circled arrows indicating a degree shift). A diagram chase with these exact sequences indicates that the image of the right-hand vertical map is isomorphic to the homology of the complex on the left-hand side.

  However, by the previous identification, the right-hand vertical map is the map
  \[
    [S^{t,w}, X\modmod\tau]_{\Fil(\Sp)} \to [S^{t,w}, \sh (X \modmod\tau)]_{\Fil(\Sp)}
  \]
  and so its image is isomorphic to $\pi_{t,w} (X\modmod\tau)$, by Proposition~\ref{prop:homotopy}.
\end{proof}

\begin{prop}
  \label{prop:spiralisderived}
  The spiral exact sequence of Proposition~\ref{prop:spiralexact} is the derived couple for the exact couple of the filtered spectrum $X$.
\end{prop}

\begin{proof}[Proof sketch.]
  The previous two propositions have identified the groups in the spiral exact sequence with the groups in the derived exact couple. Moreover, on the level of filtered spectra, the maps in the spiral exact sequence are induced by the sequence of maps:
  \[
    \sh^{-1} X \to X \to X / \sh^{-2} X \to \Sigma \sh^{-2} X
  \]
  However, the above identifications on $\pi_{t,w}$ make this the sequence of natural maps
  \[
    \begin{tikzcd}
      \im(\pi_t X^{t-w-1} \to \pi_t X^{t-w}) \ar[d] \\ 
      \im(\pi_t X^{t-w} \to \pi_t X^{t-w+1}) \ar[d] \\ 
      \im(\pi_t X^{t-w}/X^{t-w-2} \to \pi_t X^{t-w+1}/\pi_t X^{t-w-1}) \ar[d] \\ 
      \im(\pi_{t-1} X^{t-w-2} \to \pi_t X^{t-w-1})
    \end{tikzcd}
  \]
  which, while not immediate, are straightforward to verify as those defining the derived exact couple.
\end{proof}

\section{Semisynthetic spectra}
\label{sec:ssyn}

\subsection{Semisynthetic categories}

\begin{defn}
  \label{def:semisynthetic}
  Suppose that $\mc C$ is a monoidal $\infty$-category and that $Q$ is a monoid in $\Fil(\mc C)$. The \emph{semisynthetic category} associated to $Q$ is the successor of the category of left $Q$-modules:
  \[
    \SSyn_Q \simeq \LMod_Q(\Fil(\mc C))^\der.
  \]
  Here $\LMod_Q(\Fil(\mc C))$ inherits an action of the poset $\mb Z$ from $\Fil(\mc C)$.
\end{defn}

\begin{prop}
  \label{prop:ssynadjunction}
  Suppose that $\mc C$ is presentably monoidal. For a map $P \to Q$ of monoids in $\Fil(\mc C)$, the forgetful functor $\LMod_Q \to \LMod_P$ induces an adjunction
  \[
    \SSyn_Q \leftrightarrows \SSyn_P
  \]
  of semisynthetic categories.
\end{prop}

\begin{proof}
  The forgetful functor $\LMod_Q \to \LMod_P$ strictly preserves the action of $\mb Z$, and it has a shift-preserving left adjoint $X \mapsto Q \otimes_P X$. Therefore, we get an adjunction on successor categories by Proposition~\ref{prop:coherentadjunction}.
\end{proof}

\begin{defn}
  \label{def:ssynanalogue}
  The functor $\nu\co \mc C \to \SSyn_Q$ is the composite
  \[
    \mc C \xrightarrow{Q \otimes -} \LMod_Q(\Fil(\mc C)) \to \SSyn_Q.
  \]
  We refer to $\nu(X)$ as the \emph{semisynthetic analogue} of $X$.
\end{defn}

\subsection{Preservation of cofibers}

\begin{prop}
  \label{prop:tauliftshift}
  Suppose that we have a map $f\co X \to Y$ in $\mc C$ that admits a lift in the following diagram:
  \[
    \begin{tikzcd}
      X \ar[d] \ar[r,dotted] & Q^{-1} \otimes Y \ar[d] \\
      Y \ar[r] & Q^0 \otimes Y
    \end{tikzcd}
  \]
  Here the map $Y \to Q^0 \otimes Y$ is induced by the unit map of $Q$.

  Then the map $Q \otimes X \to Q \otimes Y$ lifts along $\tau$ to a map $Q \otimes X \to \sh^{-1} (Q \otimes Y)$ in $\LMod_Q$.
\end{prop}

\begin{proof}
  A map $Q \otimes X \to M$ of left $Q$-modules is, by adjunction, equivalent to a map $X \to M^0$. Therefore, the lift in the diagram determines a diagram of left $Q$-modules
  \[
    \begin{tikzcd}
      Q \otimes X \ar[r] \ar[d] & \sh^{-1} (Q \otimes Y) \ar[d,"\tau"] \\
      Q \otimes Y \ar[r,equals] & Q \otimes Y
    \end{tikzcd}
  \]
  as desired.
\end{proof}

\begin{cor}
  \label{cor:ssyncofiberpreservation}
  Suppose that $\mc C$ is stable, and that $X \to Y \to Z$ is a cofiber sequence in $\mc C$ such that the map $Z \to \Sigma X$ induces the trivial map \[
    Z \to Q^0 \otimes Z \to Q^0/Q^{-1} \otimes Z \to Q^0/Q^{-1} \otimes \Sigma X.
  \]
  Then $\nu X \to \nu Y \to \nu Z$ is still a cofiber sequence in $\SSyn_Q$.
\end{cor}

\begin{proof}
  The nullhomotopy allows us to construct a lift $Z \to Q^{-1} \otimes \Sigma X$, and hence we get a lift to a map $Q \otimes Z \to \sh^{-1}(Q \otimes \Sigma X)$. Proposition~\ref{prop:cofiberpres} then shows that the cofiber sequence $Q \otimes X \to Q \otimes Y \to Q \otimes Z$ in $\LMod_Q$ becomes a cofiber sequence in the successor $\SSyn_Q$.
\end{proof}

\subsection{Adams--Rees rings}

Consider a map $R \to E$ of strictly associative ring spectra. The fiber $I$ has the structure of a \emph{Smith ideal} \cite{hovey-smithideals}: the two unit maps $I \otimes_R I \to I$ essentially coincide. The powers of $I$, forming the $E$-Adams tower, then have the structure of a filtered ring spectrum. The following is a general version of this result.

\begin{prop}
  \label{prop:reesring}
  Suppose $R \to E$ is a map of $\mb E_1$-algebras in a stable monoidal $\infty$-category $\mc C$, with homotopy fiber $I$. Then there is an $E_1$-algebra $Rees(I)$ in $\Fil(\mc C)$ whose underlying filtered object is an Adams tower:
  \[
    \dots \to I \otimes_R I \otimes_R I \to I \otimes_R I \to I \to R \to R \to \dots
  \]
\end{prop}

\begin{proof}
  An $\mb E_1$-algebra under $R$ is equivalent to an $\mb E_1$-algebra in $R$-bimodules \cite[4.4.3.12]{lurie-higheralgebra}, and so this is a direct consequence of \cite[1.5]{gammage-hilburn-mazelgee-schobers}.
\end{proof}

\begin{defn}
  \label{def:reesring}
  We refer to a tower $\Rees(I)$ as the previous proposition as a \emph{Adams--Rees algebra} associated to the map $R \to E$, and the category $\LMod_{R_I}(\Fil(\mc C))$ as the associated category of \emph{Adams--Rees modules}.
\end{defn}

\begin{rmk}
  In particular, an Adams--Rees module includes a filtered $R$-module $\dots \to M^0 \to M^1 \to M^2 \to \dots$ with compatible lifts
  \[
    \begin{tikzcd}
      & M^{n-1} \ar[dr] \\
      I \otimes_R M^n \ar[ur, dotted] \ar[rr] && M^n
    \end{tikzcd}
  \]
  of the action of ideal $I$.
\end{rmk}

\begin{defn}
  \label{def:Essyn}
  Suppose $R \to E$ is a map of $\mb E_1$-algebras in a stable monoidal $\infty$-category $\mc C$, with homotopy fiber $I$. We will abuse notation and write $\SSyn_E = \SSyn_{\Rees(I)}$ for the semisynthetic category of left modules over the Adams--Rees algebra, and refer to it as the $E$-semisynthetic category.
\end{defn}

\begin{prop}
  \label{prop:ARcofiberpres}
  Suppose $R \to E$ is a map of $\mb E_1$-algebras in a stable monoidal $\infty$-category $\mc C$. Let $X \to Y \to Z$ be a cofiber sequence of left $R$-modules such that the map $Z \to \Sigma X$ has positive $E$-Adams filtration: it becomes nullhomotopic after tensoring over $R$ with $E$. Then $\nu X \to \nu Y \to \nu Z$ is a cofiber sequence in $\SSyn_E$.
\end{prop}

\begin{proof}
  Being $E$-null is the same as the composite map
  \[
    Z \to \Sigma X \to E \otimes_R \Sigma X
  \]
  being null, and so this is immediate from Corollary~\ref{cor:ssyncofiberpreservation}.
\end{proof}

\subsection{Modules over the cofiber of $\tau$}

If $\mb I$ is a monoidal unit in $\mc C$, the monoidal unit $\nu(\mb I)$ in $\Fil(\mc C)$ might be called $\mb I[\tau]$, and is represented by the filtered object
\[
  \dots \to 0 \to \mb I \to \mb I \to \mb I \to \mb I \to \dots
\]
The quotient $C\tau = \mb I[\tau] \modmod \tau $ might be called $\mb I[\tau] / \tau^2$, and is represented by a filtered object
\[
  \dots \to 0 \to \mb I \to \mb I \to 0 \to 0 \to \dots
\]
This has the structure of a ring object in $\Fil(\mc C)$, and modules over $C\tau$ can be thought of as ``coherent chain complex objects'' of $\mc C$ (cf. \cite{ariotta-chain}). Our goal in this section is to sketch how to make this analogy more precise.

We first recall some generalities. For a symmetric monoidal $\infty$-category $\mc D$ with an algebra object $A \in \mc D$, we define $\Gamma = A \otimes A$. The following are formal consequences:
\begin{itemize}
\item The pair $(A,\Gamma)$ has the structure of a \emph{Hopf algebroid} object in the homotopy category of $\mc D$.
\item For any $X \in \mc D$, the extension-of-scalars functor $X \mapsto A \otimes X$ from $\mc D$ to the homotopy category of $\LMod_A$ lifts to the category $\Comod_\Gamma$ of $\Gamma$-comodules.
\item The map $[X,Y] \to [A \otimes X, A \otimes Y]_{A}$ lifts to $[A \otimes X, A \otimes Y]_{\Comod_\Gamma}$.
\item If $A$ is commutative, $\Gamma$ becomes a commutative algebra in $A$-bimodules and the category of $\Gamma$-comodules obtains a symmetric monoidal structure. The extension-of-scalars functor $\mc D \to \Comod_\Gamma$ is then symmetric monoidal.
\item Any map $\theta\co A \to \Sigma^{p,q} A$, or equivalently right $A$-module map $\Gamma \to \Sigma^{p,q} A$, induces a natural operator $A \otimes X \to \Sigma^{p,q}(A \otimes X)$.
\item Identities on maps induce identities on operators.
\end{itemize}

To specialize, in $\Fil(\mc C)$, we can let $A = \nu(\mb I) / \tau$.
\begin{prop}
  In $\Fil(\mc C)$, there is an equivalence of right $A$-modules
  \[
    \Gamma = A \otimes A \simeq A \oplus \Sigma^{1,0} A.
  \]
\end{prop}

\begin{proof}
  This follows by tensoring the cofiber sequence
  \[
    \Sigma^{0,-1} \nu(\mb I) \to \nu(\mb I) \to A
  \]
  with $A$, together with the fact that the ring structure on $A$ splits the result.
\end{proof}

\begin{defn}
  The map $\partial\co A \to \Sigma^{1,0} A$ is adjoint to projection map
  \[
    A \otimes A \simeq A \oplus \Sigma^{1,0} A \to \Sigma^{1,0} A
  \]
  of right $A$-modules.
\end{defn}

This induces a natural operator $\partial_X\co A \otimes X \to \Sigma^{1,0} A \otimes X$. Moreover, $\partial$ satisfies two main identities:
\begin{itemize}
\item $\Sigma^{1,0} \partial \circ \partial = 0$
\item $\partial \circ m \simeq m \circ (\partial \otimes 1 + \beta \circ 1 \otimes \partial)$
\end{itemize}
Here $m$ is the multiplication $A \otimes A \to A$, and $\beta$ the isomorphism $A \otimes \Sigma^{1,0} A \cong \Sigma^{1,0} (A \otimes A)$.

These have three major consequences.
\begin{itemize}
\item The composite map $A \otimes X \to \Sigma^{1,0} A \otimes X \to \Sigma^{2,0} A \otimes X$ is trivial.
\item Under the equivalence $(A \otimes X \otimes Y) \simeq (A \otimes X) \otimes_A (A \otimes Y)$, we get an identification
  \[
    \partial_{X \otimes Y} \simeq (\partial_X \otimes id) + \beta(id \otimes \partial_Y).
  \]
\item The identification $[X,A \otimes Y] \simeq [A \otimes X, A \otimes Y]_A$ carries the operator $g \mapsto \partial_Y \circ g$ on the left to the operator
  \[
    g \mapsto \partial_Y \circ g - \gamma(g \circ \Sigma^{-1,0} \partial_X)
  \]
  where $\gamma$ is a natural isomorphism
  \[
    [\Sigma^{-1,0} A \otimes X, A \otimes Y]_A \cong [A \otimes X, \Sigma^{1,0} A \otimes Y].
  \]
\end{itemize}

However, for an object $Y$ in $\Fil(\mc C)$, $A \otimes Y$ is the associated graded $\gr(Y)$. Thus, we have a natural differential
\[
  \partial\co \gr(Y) \to \Sigma^{1,0} \gr(Y)
\]
satisfying $\partial^2 = 0$ and a Leibniz rule. Moreover, for $X$ and $Y$, the isomorphism
\[
  [X,\gr Y] \cong [\gr X, \gr Y]_{\gr}
\]
takes the differential on $Y$ to the natural $\Hom$-complex differential.

\begin{prop}
  Suppose $Q$ is a ring in $\Fil(\mc C)$, $X$ and $Y$ are left $Q$-modules in $\Fil(\mc C)$, making the associated graded objects $\gr(X)$ and $\gr(Y)$ are differential graded modules over $\gr(Q)$ in the homotopy category of $\mb Z$-graded objects of $\mc C$. Then there exists an isomorphism between the bigraded groups
  \[
    [X,Y \modmod \tau]_{\SSyn_Q} \cong [X \modmod \tau, Y \modmod \tau]_{\SSyn_{Q \modmod \tau}}
  \]
  and the homology of the $\Hom$-complex
  \[
    [\gr X, \gr Y]_{\gr Q}.
  \]
\end{prop}

\begin{proof}
  There is a natural commutative diagram
  \[
    \begin{tikzcd}
      \Sigma^{0,-2} Y \ar[d,"\tau^2"'] \ar[r,"\tau"]&
      \Sigma^{0,-1} Y \ar[d,"\tau^2"] \ar[dl,"\tau"']\\
      Y \ar[r,"\tau"']&
      \Sigma^{0,1} Y.
    \end{tikzcd}
  \]
  The octahedral axiom turns this into a commutative diagram
  \[
    \begin{tikzcd}
      \Sigma^{0,-1} Y/\tau \ar[dr,"i"] & \\
      &Y/\tau^2 \ar[dl,"p"] \ar[dd,"\tau"]\\
      Y/\tau \ar[dr,"\Sigma^{0,1} i"] \ar[uu, "\circ" marking] & \\
      & \Sigma^{0,1} Y/\tau^2 \ar[dl,"\Sigma^{0,1} p"]\\
      \Sigma^{0,1} Y/\tau  \ar[uu, "\circ" marking]
    \end{tikzcd}
  \]
  where the top and bottom triangles are cofiber sequences in $\LMod_Q$.

  We then apply $[X,-]_{Q}$ to this diagram. Noting that
  \[
    [X,Z/\tau]_Q \cong [X/\tau, Z/\tau]_{Q/\tau} = [\gr X, \gr Z]_{\gr Q},
  \]
  and similarly
  \[
    [X,Z/\tau^2] \cong [X \modmod \tau, Z \modmod \tau]_{Q \modmod \tau},
  \]
  this becomes a diagram
  \[
    \begin{tikzcd}
      {[\gr X, \Sigma^{0,-1} \gr Y]_{\gr Q}} \ar[dr,"i"]  & \\
      & {[X \modmod \tau, Y \modmod \tau]_{Q \modmod \tau}} \ar[dl,"p"] \ar[dd,"\tau"]\\
      {[\gr X, \gr Y]_{\gr Q}} \ar[dr,"\Sigma^{0,1} i"] \ar[uu, "\circ" marking, "\partial"] & \\
      & {[X \modmod \tau, \Sigma^{0,1} Y \modmod \tau]_{Q \modmod \tau}} \ar[dl,"\Sigma^{0,1} p"]\\
      {[\gr X, \Sigma^{0,1} \gr Y]_{\gr Q}}  \ar[uu, "\circ" marking, "\Sigma^{0,1} \partial"]
    \end{tikzcd}
  \]
  where the top and bottom triangles are long exact sequences. As in the proof of Proposition~\ref{prop:spiralisderived}, this identifies the image 
  \[
    \im(\tau\co {[X \modmod \tau, Y \modmod \tau]_{Q \modmod \tau}} \to {[X \modmod \tau, \Sigma^{0,1} Y \modmod \tau]_{Q \modmod \tau}}) \cong [X \modmod \tau, Y \modmod \tau]_{\SSyn_{(Q \modmod \tau)}}
  \]
  of the right-hand vertical map with the homology of the differential $\partial$ on $[\gr X, \gr Y]_{\gr Q}$.
\end{proof}

\begin{rmk}
  We note that Proposition~\ref{prop:indexinghell} can be alternatively recovered from the above result.
\end{rmk}

\section{Comparison to synthetic spectra}
\label{sec:syn}

\begin{thm}
  \label{thm:syntheticcomparison}
  Suppose that $E$ is an associative ring spectrum. Then the category $\Syn_E$ of synthetic spectra is equivalent to a localizing subcategory of the $E$-semisynthetic category $\SSyn_E$.
\end{thm}

\subsection{The comparison functors}

In the following we will write $\nu_{\Syn}$ and $\nu_{\SSyn}$ to clarify the distinction between these two functors.

\begin{defn}
  \label{def:ssyncomparisonfunctor}
  Let $E$ be an associative ring spectrum, and let $\mc P$ be the category of finite $E$-projective spectra. We define a functor
  \[
    \Psi\co \Fun(\mc P^\op, \Sp)\to \SSyn_E
  \]
  to be the unique hocolim-preserving extension of the functor $\nu_{\SSyn} \co \mc P \to \SSyn_E$.
\end{defn}

\begin{prop}
  The functor $\Psi$ has a right adjoint
  \[
    \Phi\co \SSyn_E \to \Fun(\mc P^\op, \Sp),
  \]
  taking a semisynthetic spectrum to the presheaf of spectra on $\mc P$ given by
  \[
    (\Phi X) (P) = \map_{\SSyn_E}(\nu_{\SSyn} P, X).
  \]
\end{prop}

\begin{proof}
  The formula in question always defines a right adjoint functor to the presheaf category. Alternatively, we can apply \cite[5.5.2.10]{lurie-htt} to construct the right adjoint, and deduce the formula by composing $\Psi$ with the Yoneda embedding.
\end{proof}

\begin{prop}
  \label{prop:comparisonholim}
  The functor $\Phi$ is holim-preserving, and takes values in the full subcategory $\Syn_E \subset \Fun(\mc P^\op, \Sp)$ of synthetic spectra.
\end{prop}

\begin{proof}
  Limits in a presheaf category are calculated objectwise, and objectwise $\Phi$ is the limit-preserving functor $\map_{\SSyn_E}(\nu_{\SSyn} P, -)$.
  
  By \cite[2.8]{pstragowski-synthetic}, an object $Y$ in the presheaf category is a synthetic spectrum if and only if, for any cofiber sequence $P \to Q \to R$ in $\mc P$ which is $E_*$-exact, we get a fiber sequence $Y(R) \to Y(Q) \to Y(P)$. However, if $P \to Q \to R$ is an $E_*$-exact sequence between finite $E_*$-projective spectra, the map $R \to \Sigma P$ is zero on $E_*$ and hence $E$-trivial by the universal coefficient isomorphism. Therefore, $\nu_{\SSyn} (P) \to \nu_{\SSyn}(Q) \to \nu_{\SSyn}(R)$ is a cofiber sequence in $\SSyn_E$ by Proposition~\ref{prop:ARcofiberpres}, and so for any $X \in \SSyn_E$ the sequence
  \[
    \map_{\SSyn_E}(\nu_{\SSyn} R, X) \to 
    \map_{\SSyn_E}(\nu_{\SSyn} Q, X) \to 
    \map_{\SSyn_E}(\nu_{\SSyn} P, X)
  \]
  is a fiber sequence, as desired.
\end{proof}

\begin{prop}
  \label{prop:comparisonhocolim}
  The functor $\Phi$ preserves homotopy colimits.
\end{prop}

\begin{proof}
  Since both categories are stable and $\Phi$ preserves fiber sequences, it also preserves cofiber sequences. Therefore, it suffices to check that $\Phi$ preserves arbitrary coproducts by \cite[4.4.2.7]{lurie-htt}.

  We note then that
  \[
    (\Phi (\oplus X_i)) (P) = \map_{\SSyn_E}(\nu_{\SSyn} P, \oplus X_i).
  \]
  However, as $P$ is a finite spectrum, $\nu_{\SSyn} P$ is compact in $\SSyn_E$ by Proposition~\ref{prop:compactpres} and so the result holds.
\end{proof}

\subsection{The left adjoint}

Because the comparison functor $\Phi$ lands in the full subcategory $\Syn_E \subset \Fun(\mc P^\op, \Sp)$, we arrive at the following.

\begin{prop}
  \label{prop:comparisonadjoint}
  The functors $\Phi$ and $\Psi$ restrict to an adjunction
  \[
    \Syn_E \rightleftarrows \SSyn_E.
  \]
\end{prop}

\begin{prop}
  \label{prop:comparisonembeddinggenerators}
  For any finite $E_*$-projective spectrum, the unit map
  \[
    \Omega^\infty \nu_{\Syn} Q \to \Phi \Psi \nu_{\Syn} Q
  \]
  is an equivalence.
\end{prop}

\begin{proof}
  By Definition~\ref{def:ssyncomparisonfunctor}, $\Psi$ is the hocolim-preserving extension of $\nu_{\SSyn}$ to $\Fun(\mc P^\op, \Sp)$, which means $\Psi \nu_{\Syn} \simeq \nu_{\SSyn}$. Therefore, we need to show that the map
  \[
    \nu_{\Syn} Q \to \Phi \nu_{\SSyn} Q
  \]
  is an equivalance in $\Fun(\mc P^\op, \Sp)$ for any $Q \in \mc P$. Taking values on any object $P$, we find that this reduces to showing
  \[
    \Map_{\Syn_E}(\nu_{\Syn}P, \nu_{\Syn} Q) \to \Map_{\SSyn_E}(\nu_{\SSyn} P, \nu_{\SSyn} Q)
  \]
  is always an equivalence.

  However, by \cite{pstragowski-synthetic} and Proposition~\ref{prop:beilinsonweight}, both sides are canonically identified with the truncation
  \[
    \Omega^\infty \tau^B_{\geq 0} \map(P, \Rees(I) \otimes Q)
  \]
  in the Beilinson $t$-structure.
\end{proof}

\begin{thm}
  \label{thm:comparisonembedding}
  The functor $\Psi\co \Syn_E \to \SSyn_E$ is fully faithful, and the essential image consists of the full subcategory of $\SSyn_E$ generated under homotopy colimits by objects of the form $\Sigma^{t,w} \nu(P)$ for $P$ a finite $E_*$-projective spectrum.
\end{thm}

\begin{proof}
  We first show full faithfulness. This is equivalent to knowing that the unit
  \[
    Y \to \Phi \Psi Y
  \]
  is always an equivalence in $\Syn_E$. By Proposition~\ref{prop:comparisonembeddinggenerators}, it is an equivalence when $Y = \nu_{\Syn} Q$, and by Proposition~\ref{prop:comparisonhocolim} both the source and target of the unit preserve hocolims in $Y$. However, $\Syn_E$ is generated under hocolims by objects $\Sigma^{t,w} \nu_{\Syn}(P)$, and so the unit map is always an equivalence.
  
  The functor $\Psi$ is hocolim-preserving, and $\SSyn_E$ is generated under hocolims by objects of the form $\Sigma^{t,w} \nu(P)$ for $P$ finite $E_*$-projective. Therefore, the essential image must be the closure of the collection of objects of the form $\Sigma^{t,w} \nu(P)$ in $\SSyn_E$ under hocolims, as desired.
\end{proof}

\begin{rmk}
  By Corollary~\ref{cor:generatedgraded}, any object in the image of $\Psi$ has a representative in $\Fil(\Sp)$ whose associated graded is a sum of terms of the form $\Sigma^{t,w} \gr(\Rees(I)) \otimes P$ for $P$ a finite $E_*$-projective. In some sense, $\Syn_E$ consists of objects with a ``cofibrant'' representative.
\end{rmk}

\begin{rmk}
  In \cite{syntheticcellular} it is shown that, when $E$ is connective, the category $\Syn_E$ is generated by the bigraded spheres. This correspondingly identifies the essential image of $\Psi$ as the full subcategory of $\SSyn_E$ generated under homotopy colimits by the bigraded spheres $S^{t,w}$.
\end{rmk}

\bibliography{../masterbib}

\begin{thebibliography}{GHMG22}

\bibitem[Ari21]{ariotta-chain}
Stefano Ariotta.
\newblock Coherent cochain complexes and {B}eilinson t-structures, with an
  appendix by {A}chim {K}rause, 2021, https://arxiv.org/abs/2109.01017.

\bibitem[BHS22]{burklund-hahn-senger-artintate}
Robert Burklund, Jeremy Hahn, and Andrew Senger.
\newblock Galois reconstruction of {A}rtin-{T}ate $\mathbb{R}$-motivic spectra,
  2022, https://arxiv.org/abs/2010.10325.

\bibitem[BL21]{skeleta}
Jonathan Beardsley and Tyler Lawson.
\newblock Skeleta and categories of algebras, 2021,
  https://arxiv.org/abs/2110.09595.

\bibitem[Boa99]{boardman-conditional}
J.~Michael Boardman.
\newblock Conditionally convergent spectral sequences.
\newblock In {\em Homotopy invariant algebraic structures ({B}altimore, {MD},
  1998)}, volume 239 of {\em Contemp. Math.}, pages 49--84. Amer. Math. Soc.,
  Providence, RI, 1999.

\bibitem[Bou03]{bousfield-cosimplicial}
A.~K. Bousfield.
\newblock Cosimplicial resolutions and homotopy spectral sequences in model
  categories.
\newblock {\em Geom. Topol.}, 7:1001--1053 (electronic), 2003.

\bibitem[Cis06]{cisinski-testcat}
Denis-Charles Cisinski.
\newblock Les pr\'efaisceaux comme mod\`eles des types d'homotopie.
\newblock {\em Ast\'erisque}, (308):xxiv+390, 2006.

\bibitem[DKS93]{dwyerkanstover-e2}
W.~G. Dwyer, D.~M. Kan, and C.~R. Stover.
\newblock An {$E^2$} model category structure for pointed simplicial spaces.
\newblock {\em J. Pure Appl. Algebra}, 90(2):137--152, 1993.

\bibitem[GH04]{goerss-hopkins-summary}
Paul~G. Goerss and Michael~J. Hopkins.
\newblock Moduli spaces of commutative ring spectra.
\newblock In {\em Structured ring spectra}, volume 315 of {\em London Math.
  Soc. Lecture Note Ser.}, pages 151--200. Cambridge Univ. Press, Cambridge,
  2004.

\bibitem[GH15]{gepner-haugseng-enriched}
David Gepner and Rune Haugseng.
\newblock Enriched {$\infty$}-categories via non-symmetric {$\infty$}-operads.
\newblock {\em Adv. Math.}, 279:575--716, 2015.

\bibitem[GHMG22]{gammage-hilburn-mazelgee-schobers}
Benjamin Gammage, Justin Hilburn, and Aaron Mazel-Gee.
\newblock Perverse schobers and 3d mirror symmetry, 2022,
  https://arxiv.org/abs/2202.06833.

\bibitem[Hei20]{heine-enriched}
Hadrian Heine.
\newblock An equivalence between enriched $\infty$-categories and
  $\infty$-categories with weak action, 2020, https://arxiv.org/abs/2009.02428.

\bibitem[Hov14]{hovey-smithideals}
Mark Hovey.
\newblock Smith ideals of structured ring spectra, 2014,
  https://arxiv.org/abs/1401.2850.

\bibitem[Law23]{laxmonoidality}
Tyler Lawson.
\newblock Lax monoidality for products of enriched higher categories, 2023,
  https://arxiv.org/abs/2305.15204.

\bibitem[Law24]{syntheticcellular}
Tyler Lawson.
\newblock Synthetic spectra are (usually) cellular, 2024,
  https://arxiv.org/abs/2402.03257.

\bibitem[Lur09]{lurie-htt}
Jacob Lurie.
\newblock {\em Higher topos theory}, volume 170 of {\em Annals of Mathematics
  Studies}.
\newblock Princeton University Press, Princeton, NJ, 2009.

\bibitem[Lur17]{lurie-higheralgebra}
Jacob Lurie.
\newblock Higher {A}lgebra.
\newblock Draft version available at:
  http://www.math.harvard.edu/\~{}lurie/papers/higheralgebra.pdf, 2017.

\bibitem[Nik16]{nikolaus-yoneda}
Thomas Nikolaus.
\newblock Stable $\infty$-operads and the multiplicative {Y}oneda lemma, 2016,
  https://arxiv.org/abs/1608.02901.

\bibitem[Pst23]{pstragowski-synthetic}
Piotr Pstr\k{a}gowski.
\newblock Synthetic spectra and the cellular motivic category.
\newblock {\em Invent. Math.}, 232(2):553--681, 2023.

\end{thebibliography}
\end{document}